\documentclass[11pt]{amsart}

\usepackage{amsfonts,epsfig}
\usepackage{latexsym}
\usepackage{amssymb}
\usepackage{amsmath}
\usepackage{amsthm}
\usepackage{graphics}
\usepackage{verbatim}
\usepackage{enumerate}
\usepackage[all]{xy}
\usepackage{array,booktabs,ctable}
\usepackage[backref]{hyperref}
\usepackage[capitalize]{cleveref}
\usepackage{color}
\usepackage{hhline}
\usepackage{xcolor}
\usepackage{colortbl}
\definecolor{mylinkcolor}{rgb}{0.8,0,0}
\definecolor{myurlcolor}{rgb}{0,0,0.8}
\definecolor{mycitecolor}{rgb}{0,0,0.8}
\hypersetup{colorlinks=true,urlcolor=myurlcolor,citecolor=mycitecolor,linkcolor=mylinkcolor,linktoc=page,breaklinks=true}
\usepackage{bigstrut}

 \usepackage{tabu}

\newtheorem{defn}{Definition}[section]
\newtheorem{definition}[defn]{Definition}

\newtheorem{corollary}[defn]{Corollary}
\newtheorem{lemma}[defn]{Lemma}

\newtheorem{thm}[defn]{Theorem}

\newtheorem{cor}[defn]{Corollary}
\newtheorem{prop}[defn]{Proposition}
\newtheorem{proposition}[defn]{Proposition}

\theoremstyle{definition}

\newtheorem{remark}[defn]{Remark}
\newtheorem{question}[defn]{Question}
\newtheorem{example}[defn]{Example}

\newcommand{\QQ}{\mathbb Q}
\newcommand{\Qbar}{\overline{\QQ}}
\newcommand{\ZZ}{\mathbb Z}
\renewcommand{\NN}{\mathbb N}
\newcommand{\FF}{\mathbb F}

\newcommand{\PP}{\mathbb P}

\newcommand{\field}[1]{\mathbb{#1}}  
\newcommand{\Q}{\field{Q}} 
\newcommand{\N}{\field{N}} 
\newcommand{\Z}{\field{Z}} 
\newcommand{\F}{\field{F}} 

\newcommand{\SL}{\operatorname{SL}}

\newcommand{\Gal}{\operatorname{Gal}}

\newcommand{\GL}{\operatorname{GL}}

\newcommand{\tor}{\mathrm{tors}}
\newcommand{\GenA}{A_4^\infty}
\newcommand{\QA}{\QQ(\GenA)}

\renewcommand{\Im}{{\rm Im}\,}
\newcommand{\eval}{{\rm eval}}
\DeclareMathOperator{\aut}{Aut}
\newcommand{\set}[1]{\left\lbrace #1 \right\rbrace}

\DeclareMathOperator{\Hom}{Hom}

\begin{document}

\title[Groups of generalized $G$-type]{Groups of generalized $G$-type and applications to torsion subgroups of rational elliptic curves over infinite extensions of $\QQ$}

\author{Harris B. Daniels}
\address{Department of Mathematics and Statistics, Amherst College, Amherst, MA 01002}
\email{hdaniels@amherst.edu}
\urladdr{\url{https://hdaniels.people.amherst.edu/}}

\author{Maarten Derickx}
\address{Department of Mathematics, MIT 2-106, 77 Massachusetts Ave, Cambridge, MA 02139}
\email{maarten@mderickx.nl}
\urladdr{\url{http://www.maartenderickx.nl/}}
\thanks{The second author was supported by Simons Foundation grant 550033}

\author{Jeffrey Hatley}
\address{Department of Mathematics, Union College, Schenectady, NY 12308}
\email{hatleyj@union.edu}
\urladdr{\url{http://www.math.union.edu/~hatleyj/}}

\maketitle

\begin{abstract}
Recently there has been much interest in studying the torsion subgroups of elliptic curves base-extended to infinite extensions of $\QQ$. In this paper, given a finite group $G$, we study what happens with the torsion of an elliptic curve $E$ over $\QQ$ when changing base to the compositum of all number fields with Galois group $G$. We do this by studying a group theoretic condition called generalized $G$-type, which is a necessary condition for a number field with Galois group $H$ to be contained in that compositum. In general, group theory allows one to reduce the original problem to the question of finding rational points on finitely many modular curves. To illustrate this method we completely determine which torsion structures occur for elliptic curves defined over $\QQ$ and base-changed to the compositum of all fields whose Galois group is $A_4$.
\end{abstract}

\section{introduction}

At the foundation of the study of arithmetic geometry is the Mordell-Weil Theorem, which states that given an elliptic curve $E$ defined over a number field $K$, the set of $K$-rational points on $E$, denoted $E(K)$, forms a finitely generated abelian group. That is to say, there exists an integer $r\geq 0$ and a finite abelian group, which we will denote $E(K)_\tor$, such that
\[
E(K)\simeq \ZZ^r \oplus E(K)_\tor.
\]
The positive integer $r$ is called the rank of $E$ over $K$, while $E(K)_\tor$ is called the torsion subgroup of $E$ over $K$. Many interesting and difficult questions can be asked about the ranks of elliptic curves, but the central objects of study in this paper will be the torsion subgroups of elliptic curves.

A natural first question one can ask in this direction is whether there are restrictions on what groups can arise as the torsion subgroup of an elliptic curve. This question is brought into sharper focus by a theorem which states that, for every number field $K$, there are positive integers $a$ and $b$ such that $E(K)_\tor \simeq \ZZ/a\ZZ\oplus\ZZ/ab\ZZ$ for any elliptic curve $E$ defined over $K$. In order to obtain more precise results, one needs to put restrictions on the field of definition of the elliptic curve. The first such result is the following that restricts to the case when $K=\QQ$.

\begin{thm}[Mazur \cite{Mazur1}]\label{thm:mazur}
Let $E/\QQ$ be an elliptic curve. Then
\[
E(\QQ)_\tor \simeq \begin{cases}
 \ZZ/M\ZZ  &   1\leq M \leq 10 \hbox{ or } M=12, \hbox{ or}\\
\ZZ/2\ZZ\oplus\ZZ/2M\ZZ  &  1\leq M\leq 4.
 \end{cases}
\]
\end{thm}

More results of this nature can be found in \cite{kenmom, kamienny}, where the torsion structures that occur when $K$ is allowed to be a quadratic extension of $\QQ$ are completely classified. Since then much work has been put into classifying the torsion structures of elliptic curves defined over cubic extensions of $\QQ$ with results recently announced by Etropolski, Morrow, and Zureick-Brown, and independently by the second-named author. The case where $[K:\QQ] = 4$ is still completely wide open.

Another approach is to start with an elliptic curve defined over $\QQ$ and consider what torsion subgroups occur when it is base-extended to a field $K / \QQ$. The cases when $E$ is defined over $\QQ$ and $K$ is a number field of degree $d$ with $d= 2,3, 4 ,5, 7$ or $d$ is not divisible by $2,3,5,7$ have been completely settled in
\cite{GonJim-Tor,najman-cubic,Chou,GonJim,GonNajman}. There have been a number of papers that consider the question of what torsion structures occur over a fixed \textit{infinite} extension of $\QQ$. For example, in \cite{laska} and \cite{fujita1,fujita2} all the possible torsion structures that occur when $E/\QQ$ is base-extended to the compositum of all quadratic extensions of $\QQ$ are classified. More recently, in \cite{Q(3)} all of the torsion structures that occur when $E/\QQ$ is base-extended to the compositum of all degree 3 extensions of $\QQ$ are classified. It is worth noting at this point that when working over an infinite extension of $\QQ$, we no longer have the Mordell-Weil Theorem to ensure that the torsion subgroup remains finite. Because of this we have to make sure that we choose an infinite extension of $\QQ$ carefully.

In this paper, we provide a general framework for studying the torsion subgroups of rational elliptic curves over certain infinite extensions of $\QQ$. Before proceeding we remind the reader of the following definition.

\begin{defn}
For $n \in \N$ and groups $G_1, \ldots, G_n$, let $\pi_i\colon G_1\times \cdots \times G_n \to G_i$ be the standard projection maps. A subgroup $K$ of $G_1\times \cdots \times G_n$ is a {\bfseries subdirect product} of $G_1, \ldots, G_n$ if for all $i$, $\pi_i(K) = G_i$.
\end{defn}

\noindent
With this and inspired by \cite{Q(3),D4}, we give the following definition.

\begin{defn}\label{def:Q(G^infty)}
Let $G$ be a transitive subgroup of $S_d$ for some $d\geq 2$. We say that a finite group $H$ is of {\bfseries generalized $G$-type} if it is isomorphic to a quotient of a subdirect product of transitive subgroups of $G$. Given a number field $K/\QQ$ and its Galois closure $\widetilde{K}$, we say that $K/\QQ$ is of {\bfseries generalized $G$-type} if $\Gal(\widetilde{K}/\QQ)$ is a group of generalized $G$-type. Let $\QQ(G^\infty)$ be the compositum of all fields that are of generalized $G$-type.
\end{defn}

\noindent
Immediately from this definition we get the following lemma.
\begin{lemma}\label{lem:finite_root_of_unity_Ginfty}
Let $G$ be a finite group. The extension $\QQ(G^\infty)/\QQ$ is a Galois extension and the set $S\subseteq \mathbb{N}$ consisting of the integers $n$ such that $\QQ(\zeta_n) \subseteq \QQ(G^\infty)$ is finite.
\end{lemma}
\begin{proof} Let $\lambda(n)$ be the exponent of $\mathrm{Gal}(\QQ(\zeta_n)/\QQ) \simeq \left( \ZZ / n \ZZ \right)^\times$; then $\lambda(2^e)=2^{e-2}$ for $e \geq 3$ and $\lambda(p^e)=(p-1)p^{e-1}$ for odd primes $p>2$ and $e \geq 1$. If $\QQ(\zeta_n) \subseteq \QQ(G^\infty)$, then $\QQ(\zeta_n)$ must be of generalized $G$-type, and hence we must have $\lambda(n) \mid \exp(G)$, where $\exp(G)$ is the exponent of $G$. As $\exp(G)$ is finite, the set $S \subseteq \set{n \in \N : \lambda(n) | \mathrm{exp}(G)}$ is clearly finite as well, and the lemma follows.
\end{proof}
With this lemma we have that for any finite group $G$ we can now apply the following theorem to to the torsion subgroups of rational elliptic curves base-extended to $\QQ(G^\infty)$.

\begin{thm}\cite[Theorem 4.1]{Q(3)} \label{thm:finite_over_F}
Let $E/\Q$ be an elliptic curve and let $F$ be a (possibly infinite) Galois extension of~$\Q$ that contains only finitely many roots of unity. Then $E(F)_\tor$ is finite.  Moreover, there is a uniform bound $B$, depending only on $F$, such that $\#E(F)_\tor\le B$ for every elliptic curve $E/\Q$.
\end{thm}

In light of Theorem \ref{thm:finite_over_F}, it is natural to ask the following question.
\begin{question}\label{ques:main_question}
Given a finite group $G$, what groups (up to isomorphism) occur as the torsion subgroup of an elliptic curve $E/\QQ$ base extended to $\QQ(G^\infty)$?
\end{question}

We start the study of this question in Section \ref{sec:gen-type} by first introducing the notions of {\it weak} and {\it strong} generalized $G$-type. Armed with these two new concepts, we find computable necessary and sufficient conditions to determine if a given finite group $H$ is of generalized $G$-type. These necessary and sufficient conditions drastically generalize the conditions used in \cite[Lemma 3.2]{Q(3)} and \cite[Lemma 3.2]{D4} where the cases of $G=S_3$ and $D_4$ are studied.

Once these necessary and sufficient conditions are established, Question \ref{ques:main_question} reduces to classifying rational points on modular curves. In fact, since there is a computable uniform bound on $\# E(\QQ(G^\infty))_\tor$  depending only on $\QQ(G^\infty)$ (and so really depending only on $G$), Question \ref{ques:main_question} is reduced to finding the rational points on a {\it finite and computable} list of modular curves depending on $G$. To explicitly illustrate our method we fully answer Question \ref{ques:main_question} in the case when $G=A_4$ by proving the following theorem.

\begin{thm}\label{thm:main}
Let $E/\QQ$ be an elliptic curve. The torsion subgroup $E(\QQ(A_4^\infty))_\tor$ is finite and
\[
E(\QQ(A_4^\infty))_\tor\simeq
\begin{cases}
\ZZ/M\ZZ & \hbox{with }M = 1,3,5,7,9,13,15,21 \hbox{ or}\\
\ZZ/2\ZZ \oplus \ZZ/2M\ZZ & \hbox{with }1\leq M \leq 9 \hbox{ or}\\
\ZZ/3\ZZ \oplus \ZZ/3M\ZZ & \hbox{with }M = 1,3 \hbox{ or}\\
\ZZ/4\ZZ \oplus \ZZ/4M\ZZ & \hbox{with } 1 \leq M \leq 4,\ M = 7\hbox{ or}\\
\ZZ/6\ZZ \oplus \ZZ/6\ZZ &  \hbox{ or}\\
\ZZ/8\ZZ \oplus \ZZ/8\ZZ &

\end{cases}
\]
All but $4$ of the $26$ torsion structures listed above occur for infinitely many $\Qbar$-isomorphism classes of elliptic curves $E/\QQ$. The torsion structures that occur finitely often are
\[
\ZZ / 21 \ZZ, \quad \ZZ / 15\ZZ, \quad \ZZ / 2 \ZZ \oplus \ZZ / 14 \ZZ, \quad \text{and} \quad \ZZ / 3 \ZZ \oplus \ZZ / 9 \ZZ
\]
which occur for $4,2,2$, and $1$ $\Qbar$-isomorphism classes respectively.
\end{thm}

The proof of Theorem \ref{thm:main} gives a concrete demonstration of the general strategy for determining the torsion subgroups for curves which have been base-extended to $\QQ(G^\infty)$ for more general $G$. Thus, the results of this paper provide a framework for determining the torsion subgroups for elliptic curves $E/\QQ$ base-extended to a large family of infinite extensions of $\QQ$. We note here that we do not have an algorithm to classify these groups up to isomorphism since there is not an algorithm to find all the rational points on a general modular curve, although in practice this is possible.

To finish off the classification over $\QQ(A_4^\infty)$, in Table \ref{tab:Examples} we give elliptic curves of minimal conductor with each of the 26 possible torsion structures when base-extended to $\QQ(A_4^\infty)$. Further, in Table \ref{tab:TParams} we find a complete list of (possibly constant) rational maps that classify the $j$-invariants of elliptic curves which contain a subgroup isomorphic to one of the possible 26 torsion structures. More information about Table \ref{tab:TParams} can be found in Section \ref{sec:TParam}.

\begin{table}[h!]
\begin{center}
\renewcommand{\arraystretch}{1}
\begin{tabular}{|l|l||l|l|}\hline
	$E/\QQ$ & $E(\QQ(A_4^\infty))_\tor$ & $E/\QQ$ & $E(\QQ(A_4^\infty))_\tor$\\
	\hline\bigstrut[t]
\href{http://www.lmfdb.org/EllipticCurve/Q/11a2}{\texttt{11a2}}  & $\{\mathcal{O}\}$ &
\href{http://www.lmfdb.org/EllipticCurve/Q/30a1}{\texttt{30a1}}  & $\ZZ/2\ZZ\oplus \ZZ/12\ZZ$ \\
\href{http://www.lmfdb.org/EllipticCurve/Q/44a1}{\texttt{44a1}}  & $\ZZ/3\ZZ$ &
\href{http://www.lmfdb.org/EllipticCurve/Q/49a1}{\texttt{49a1}}  & $\ZZ/2\ZZ\oplus \ZZ/14\ZZ$ \\
\href{http://www.lmfdb.org/EllipticCurve/Q/11a1}{\texttt{11a1}}  & $\ZZ/5\ZZ$ &
\href{http://www.lmfdb.org/EllipticCurve/Q/210e1}{\texttt{210e1}}  & $\ZZ/2\ZZ\oplus \ZZ/16\ZZ$ \\
\href{http://www.lmfdb.org/EllipticCurve/Q/26b1}{\texttt{26b1}}  & $\ZZ/7\ZZ$ &
\href{http://www.lmfdb.org/EllipticCurve/Q/14a3}{\texttt{14a3}}  & $\ZZ/2\ZZ\oplus \ZZ/18\ZZ$ \\
\href{http://www.lmfdb.org/EllipticCurve/Q/19a2}{\texttt{19a2}}  & $\ZZ/9\ZZ$ &
\href{http://www.lmfdb.org/EllipticCurve/Q/19a1}{\texttt{19a1}}  & $\ZZ/3\ZZ\oplus \ZZ/3\ZZ$ \\
\href{http://www.lmfdb.org/EllipticCurve/Q/147b1}{\texttt{147b1}}  & $\ZZ/13\ZZ$ &
\href{http://www.lmfdb.org/EllipticCurve/Q/27a1}{\texttt{27a1}}  & $\ZZ/3\ZZ\oplus \ZZ/9\ZZ$ \\
\href{http://www.lmfdb.org/EllipticCurve/Q/50a3}{\texttt{50a3}}  & $\ZZ/15\ZZ$ &
\href{http://www.lmfdb.org/EllipticCurve/Q/17a1}{\texttt{17a1}}  & $\ZZ/4\ZZ\oplus \ZZ/4\ZZ$ \\
\href{http://www.lmfdb.org/EllipticCurve/Q/162b1}{\texttt{162b1}}  & $\ZZ/21\ZZ$ &
\href{http://www.lmfdb.org/EllipticCurve/Q/15a2}{\texttt{15a2}}  & $\ZZ/4\ZZ\oplus \ZZ/8\ZZ$ \\
\href{http://www.lmfdb.org/EllipticCurve/Q/46a1}{\texttt{46a1}}  & $\ZZ/2\ZZ\oplus \ZZ/2\ZZ$ &
\href{http://www.lmfdb.org/EllipticCurve/Q/30a2}{\texttt{30a2}}  & $\ZZ/4\ZZ\oplus \ZZ/12\ZZ$ \\
\href{http://www.lmfdb.org/EllipticCurve/Q/17a3}{\texttt{17a3}}  & $\ZZ/2\ZZ\oplus \ZZ/4\ZZ$ &
\href{http://www.lmfdb.org/EllipticCurve/Q/210e2}{\texttt{210e2}}  & $\ZZ/4\ZZ\oplus \ZZ/16\ZZ$ \\
\href{http://www.lmfdb.org/EllipticCurve/Q/20a1}{\texttt{20a1}}  & $\ZZ/2\ZZ\oplus \ZZ/6\ZZ$ &
\href{http://www.lmfdb.org/EllipticCurve/Q/1922c1}{\texttt{1922c1}}  & $\ZZ/4\ZZ\oplus \ZZ/28\ZZ$ \\
\href{http://www.lmfdb.org/EllipticCurve/Q/15a5}{\texttt{15a5}}  & $\ZZ/2\ZZ\oplus \ZZ/8\ZZ$ &
\href{http://www.lmfdb.org/EllipticCurve/Q/14a1}{\texttt{14a1}}  & $\ZZ/6\ZZ\oplus \ZZ/6\ZZ$ \\
\href{http://www.lmfdb.org/EllipticCurve/Q/66c1}{\texttt{66c1}}  & $\ZZ/2\ZZ\oplus \ZZ/10\ZZ$ &
\href{http://www.lmfdb.org/EllipticCurve/Q/15a1}{\texttt{15a1}}  & $\ZZ/8\ZZ\oplus \ZZ/8\ZZ$ \\
	\hline
\end{tabular}
\end{center}
\caption{Examples of minimal conductor for each possible torsion structure over $\QQ(A_4^\infty)$}  \label{tab:Examples}
\end{table}

Once we have settled the classification of torsion over $\QQ(A_4^\infty)$, it is not hard to determine which torsion structures arise over the compositum of all $A_4$-extensions of $\QQ$, which we denote by $\QQ_{A_4}$. As we explain in Section \ref{sec:compositum}, since we have a proper containment $\QQ_{A_4} \subseteq \QQ(A_4^\infty)$, this comes down to determining which torsion subgroups from Theorem \ref{thm:main} still arise over the smaller field $\QQ_{A_4}$. The fact that $\QQ_{A_4}$ contains no subextensions which are quadratic over $\QQ$ plays a major role. We obtain the following theorem.

\begin{thm}\label{thm:main_strong-intro}
Let $E/\QQ$ be an elliptic curve. The torsion subgroup $E(\QQ_{A_4})_\tor$ is finite and
\[
E(\QQ_{A_4})_\tor\simeq
\begin{cases}
\ZZ/M\ZZ & \hbox{with } 1\leq M \leq 10 \hbox{ or } M = 12, 13, 14, 18, 21 \hbox{ or}\\
\ZZ/2\ZZ \oplus \ZZ/2M\ZZ & \hbox{with }1\leq M \leq 4 \hbox{ or } M = 7.\\
\end{cases}
\]
\end{thm}

We conclude in Section \ref{sec:torsion-cubic} with an observation about torsion over cyclic cubic fields. In particular, letting $\QQ(C_3^\infty)$ denote the compositum of all cyclic cubic extensions of $\QQ$, we obtain for free a classification of all torsion subgroups arising as $E(\Q(C_3^\infty))_\tor$ for an elliptic curve $E/\QQ$; see Corollary \ref{cor:final-corollary-cubic}.

All of the computations in this article were done with the help of the computer program Magma \cite{Magma} and the code has been made available at \cite{A4Code}.

\section{Groups and fields of generalized $G$-type}\label{sec:gen-type}

Starting in Section \ref{sec:strong-and-weak-G-types}, we will make a general study of groups and fields of generalized $G$-type by introducing the concepts of weak and strong $G$-type. With these we find necessary and sufficient conditions to show a group is of generalized $G$-type. After doing this general study, in Section \ref{sec:just-A4} we will examine the particular case when $G = A_4$ in order to gain more specific information to prove Theorem \ref{thm:main}.

\subsection{Strong and weak G-types}\label{sec:strong-and-weak-G-types}

\begin{definition} Let $G$ be a group. Then a group $H$ is of weak $G$-type if there is an integer $n$ such that $H$ is isomorphic to a subquotient of $G^n$. Further, $H$ is of strong $G$-type if there is an integer $N$ such that $H$ is isomorphic to a quotient of a subdirect product of $G^N$.
\end{definition}

The notion of generalized $G$-type is dependent on a choice of an embedding of $G$ in $S_d$, while in contrast the notions of strong and weak $G$-type depend only on the group $G$ itself as an abstract group. Immediately from these definitions we get the following lemma.

\begin{lemma} Let $G,H$ and $I$ be groups with $G \subseteq S_d$ a transitive subgroup. \label{lem:easy_list}
\begin{enumerate}
\item[\rm (1)] One has the implications: \begin{center} $H$ is of strong $G$-type $\Rightarrow$ $H$ is of generalized $G$-type $\Rightarrow$ $H$ is of weak $G$-type.\end{center}

\item[\rm (2)] If all subgroups of $G$ are of generalized $G$-type then $H$ is of generalized $G$-type if and only if it is of weak $G$-type.

\item[\rm (3)] If all transitive subgroups of $G$ are of strong $G$-type then $H$ is of strong $G$-type if and only if it is of generalized $G$-type. In particular, if $G = S_d$ then the notions of strong $G$-type and generalized $G$-type agree.

\item[\rm (4)] Let $T_1, \ldots T_n$ be the transitive subgroups, (respectively all subgroups) of $G$. Then the notions of strong $T_1 \times \ldots \times T_n$-type and generalized $G$-type (respectively weak $G$-type) coincide.

\item[\rm (5)] The notion of weak $G$-type is closed under taking subgroups, quotients and products, i.e. if $H$ is of weak $G$-type then so is any subgroup of $H$ and any quotient of $H$, and if $H'$ is another subgroup of weak $G$-type then so is $H \times H'$.

\item[\rm (6)] The notions of strong and generalized $G$-type are closed under taking quotients and subdirect products, i.e. if $H$ is of strong (respectively generalized) $G$-type then so is every quotient of $H$, and if $H'$ is of strong (respecitively generalized) $G$-type then so is any subdirect product of $H \times H'$.\label{lem:easy_list6}

\item[\rm (7)] The notions of being of weak, generalized and strong $G$-type are transitive, i.e. if $H$ is of weak (resp. generalized, strong) $G$-type and $I$ is of weak (resp. generalized, strong) $H$-type then $I$ is of weak (resp. generalized, strong) $G$-type.

\item[\rm (8)] Let $I_1,\ldots,I_n$ be finite groups and $H$ a subdirect product of $I_1 \times \cdots \times I_n$. Then $H$ is of weak (resp. generalized, strong) $G$-type if and only if $I_1, \ldots, I_n$ are.
\end{enumerate}
\end{lemma}

\begin{remark}
Using part 2 of the above lemma one can show that the notions of generalized $G$-type and weak $G$-type agree for $D_4,A_4, S_4 \subseteq S_4$. Using part $3$ one can show that the notions of strong $D_4$-type and generalized $D_4$-type also agree. However, this is not true for $A_4$; see Lemma \ref{lem:strong_Dpq}.
\end{remark}

For the rest of this section we will denote the free group with $k$ generators by $F_k$. Also, if $H$ is a group and $h_1,\ldots,h_k \in H$ are elements, then the morphism $F_k \to H$ which sends the $i$-th generator of $F_k$ to $h_i$ is denoted by $\eval_{h_1,\ldots,h_k}$.

\begin{definition}
Let $G$ be a finite group and $(H,h_1,\ldots,h_k)$ be a group of weak (resp. generalized or strong) $G$-type with $k$ distinguished elements. Then $(H,h_1,\ldots,h_k)$ is called a universal group of weak (resp. generalized or strong) $G$-type if, for every other group of weak (resp. generalized or strong) $G$-type with $k$ distinguished elements $(I,i_1,\ldots,i_k)$, there exists a unique $\phi : H \to I$ such that $\phi(h_j)=i_j$ for all $1 \leq j \leq k$.
\end{definition}

\begin{lemma}\label{lem:univ-G-type} Let $G$ be a group. Define
\begin{align*}
\pi_{G,k}:  F_k & \to G^{\Hom(F_k,G)} \\
x & \mapsto (f(x))_{f \in \Hom(F_k,G)}
\end{align*}
and let $G^{univ}_k = \pi_{G,k}(F_k)$. Then $\left( G^{univ}_k,\pi_{G,k}(x_1),\ldots,\pi_{G,k}(x_k) \right)$ is a universal group of weak $G$-type with $k$ distinguished elements.
\end{lemma}

Notice that if $G$ is finite then it is possible to compute $G_k^{univ}$ explicitly as a subgroup of $G^{\Hom(F_k,G)}$. In practice however one does not need a copy of $G$ for every element of $\Hom(F_k,G)$, since the kernel of a morphism does not change if we compose it with an automorphism of its range. More generally, if $f,g \in \Hom(F_k,G)$ are such that $\ker(g) \subseteq \ker(f)$, then the image of $G_k^{univ}$ maps isomorphically onto the projection where the copy of $G$ corresponding to $f$ is omitted.
\begin{corollary}\
\begin{enumerate}
\item[\rm (1)] Let $G,H$ be finite groups and suppose that $H$ is $k$-generated. Then $H$ is of weak $G$-type if and only if $H$ is isomorphic to a quotient of $G^{univ}_k$.
\item[\rm (2)] If $G$ can be generated by $k$ elements then the notion of weak $G$-type and weak $G_k^{univ}$-type agree.
\end{enumerate}
\end{corollary}

\begin{remark}\label{rem:G-type-algorithm}
Since $G^{univ}_k$ is computable and it is possible to enumerate all normal subgroups of $G^{univ}_k$, this also gives a (possibly very slow) algorithm to check whether a $k$-generated group $H$ is of weak $G$-type. One just needs to check whether there is a normal subgroup $N$ of $G^{univ}_k$ such that $H \simeq G^{univ}_k/N$. How practical this algorithm is depends on how easily one can compute $G^{univ}_k$ in practice, or whether one can give a nice theoretical description of $G^{univ}_k$ for the given group $G$. For example, if it is easy to write down a finite set of generators $g_1,\ldots,g_n$ for $\ker(\pi_{G,k})$, then determining whether $H$ is of generalized $G$-type just comes down to checking whether $\eval_{h_1,\ldots,h_k}(g_i)=0$ for all $1\leq i \leq n$.
\end{remark}

\begin{remark}
Let $G$ be a group, then \cref{lem:univ-G-type} shows that a universal $k$-generated group of weak $G$-type exists, but its true power lies in that it gives a concrete description of $G_k^{univ}$. The equivalent of \cref{lem:univ-G-type} obtained by replacing weak $G$-type with strong $G$-type and replacing $\Hom(F_k,G)$ by the surjective homomorphisms does not hold. However the universal strong $G$-type group still exists. Indeed this is a quite formal consequence of the fact that the notion of strong $G$-type is closed under quotients and subdirect products, so that one can describe the universal strong $G$-type group as the image of $G_k^{univ}$ (or $F_k$) under the product of all surjective maps to strong $G$-type groups. What goes wrong with the naive generalization of \cref{lem:univ-G-type} can be seen by computing that $\Z/2\Z$ is the universal strong $G$-type group with $1$ distinguished element where $G$ is $\Z/2\Z \times \Z/2\Z$ or $\SL_2(\F_p)$ with $p \geq 5$.
\end{remark}

\begin{definition}\
\begin{enumerate}
	\item[\rm (1)] An element $x \in F_k$ in the kernel of $\pi_{G,k}$ is called a weak $G$-type relation of rank $k$, or alternatively, it is said that $G$ satisfies the relation $x$. In this case $I_{G,k} := \ker (\pi_{G,k})$ is called the group of $G$-type relations, so that $G^{univ}_k = F_k/I_{G,k}$.
	\item[\rm (2)] If $k_1,\ldots, k_n$ are integers, $x_1,\ldots,x_n$ are in $F_{k_1},\ldots F_{k_n}$ respectively and $(G,g_1,\ldots,g_k)$ is a group with $k$ distinguished elements such that $G$ is of $x_1,\ldots,x_n$-type, then the tuple $(G,g_1,\ldots,g_k)$ is called a universal group with $k$-generators satisfying relations $x_1,\ldots,x_n$ if for every other group $(H,h_1,\ldots,h_k)$ of $x_1,\ldots,x_n$-type with $k$ distinct elements there is a unique homomorphism $\phi : G \to H$ such that $\phi(g_i)=h_i$ for $1\leq i \leq k$.
	\item[\rm (3)] If $k'_1,\ldots, k'_{n'}$ is another set of integers and $x'_1,\ldots,x'_{n'} \in F_{k_1},\ldots F_{k_n}$ is another set of relations, then $x_1, \ldots, x_n$ are said to generate the relations $x'_1,\ldots x'_{n'}$ if every group $H$ of $x_1, \ldots, x_n$-type also of $x'_1,\ldots x'_{n'}$-type. The $x_1, \ldots, x_n$ and $x'_1,\ldots x'_{n'}$ are called equivalent if they both generate each other.
\end{enumerate}
\end{definition}

\begin{example}\
\begin{itemize}
\item If $G$ is a group of exponent $n$ then $x_1^n \in F_1$ is a weak $G$-type relation of rank $1$.
\item If $G$ is an abelian group then $[x_1,x_2] := x_1^{-1}x_2^{-1}x_1x_2 \in F_2 $  is a weak $G$-type relation of rank $2$.
\item More general if $G$ is of nilpotency class $k$ then $[x_1,[x_2,[x_3,\ldots[x_{k-1},x_k]]]]$ is a weak $G$-type relation.
\item If $G$ is an abelian group of exponent $n$ then the notions of strong, weak and generalized $G$-type agree, and a group $H$ is of weak $G$-type if and only if it satisfies the relations $x_1^n$ and $[x_1,x_2]$, i.e.  $x_1^n$ and $[x_1,x_2]$ generate the weak $G$-type relations.
\end{itemize}
\end{example}

The above language can be used to restate the classification of groups of generalized $D_4$-type given in \cite{D4}. There the first named author shows that a group $G$ is of generalized $D_4$ type if and only if it has exponent 4 and is of nilpotency class at most two. Since all subgroups of $D_4$ are of generalized $D_4$-type one has that the notions of generalized $D_4$-type and weak $D_4$-type agree and so a group is of generalized-$D_4$ type if and only if it satisfies the relations $x_1^4$, $[x_1,[x_1,x_2]]]$ and $[x_2,[x_1,x_2]]]$.


\subsection{Semidirect products of elementary abelian $p$ and $q$ groups.}\label{sec:semidirect}

Let $N$, $H$ be groups and $\rho : H \to \aut N$ be a group action of $H$ on $N$. Suppose that $f : N \rtimes H \to G$ is a group morphism. Then the action of $H$ on $f(N)$ via conjugation in $G$ turns $f|_N : N \to f(N)$ into a morphism of group actions of $H$. More precisely, for all $n\in N$ and $h \in H$ one has $f(\rho(h)(n)) = f(h)f(n)f(h)^{-1}$. In fact, this is the universal property of the semidirect product.

\begin{proposition}\cite[Proposition 27]{Bourbaki}
Let $N, H$ and $G$ be groups and $\rho : H \to \aut N$ be a group action of $H$ on $N$ and suppose that $f_1 : N  \to G$ and $f_2 : H \to G$ are two group homomorphisms such that $f_1(\rho(h)(n)) = f_2(h)f_1(n)f_2(h)^{-1}$. Then there exists a unique $f : N \rtimes H \to G$ such that $f|_N = f_1$ and $f|_H = f_2$.
\end{proposition}

To simplify the notion that is to come we give the following definition.

\begin{defn}
Let $G$ be a finite group and $p$ a prime. Then we let $\lambda_p(G)$ be the 1st term in the lower central $p$-series of $G$; that is,
\[
\lambda_p(G) = [G,G]G^p.
\]
\end{defn}

Let $p$ and $q$ be distinct primes, let $\zeta_p$ be a $p$-th root of unity, let $\rho_{p,q} : \ZZ/p\ZZ \to \F_q(\zeta_p)^*$ denote the action of $\ZZ/p\ZZ$ on $\F_q(\zeta_p)$ defined by $\rho_{p,q}(i) = \zeta_p^i$, and let $D_{p,q} :=  \F_q(\zeta_p) \rtimes_\rho \ZZ/p\ZZ$ denote the corresponding semidirect product of $\ZZ/p\ZZ$ and $\F_q(\zeta_p)$ as groups.
The following is a generalization of Proposition \ref{prop:A4_Classification} and also of \cite[Lemma 3.2]{Q(3)}.
\begin{prop}\label{prop:weak_Dpq} Let $p,q$ be distinct primes and $H$ a finite group. Then the following are equivalent:
	\begin{enumerate}
		\item[\rm (1)] $H$ is isomorphic to a subgroup of $D_{p,q}^k$ for some integer $k$.
		\item[\rm (2)] $H$ is of weak $D_{p,q}$-type.
		\item[\rm (3)] $\lambda_q(\lambda_p(H)) = \set{e}$
		\item[\rm (4)] $H$ is isomorphic to $(\ZZ/q\ZZ)^n \rtimes_\rho (\ZZ/p\ZZ)^m$ for some action $\rho : (\ZZ/p\ZZ)^m \to \GL_n( \ZZ/q\ZZ)$.
	\end{enumerate}
\end{prop}

\begin{proof}
The implication $(1) \Rightarrow (2)$  is by definition, $(2) \Rightarrow (3)$ follows because $\lambda_q(\lambda_p(D_{p,q})) = \set{e}$, and $(3) \Rightarrow (4)$ follows from the Schur-Zassenhaus Theorem, so the only thing left to prove is $(4) \Rightarrow (1)$.

	View $V := (\Z/q\Z)^n$ as a vector space over $\F_q$ so that $\rho$ can be seen as an $\F_q$-linear representation. Since the order of $(\Z/p\Z)^m$ is coprime to $q$, $V$ decomposes into a direct sum of irreducible representations of $(\Z/p\Z)^m$. Let $V_{triv}\simeq (\Z/q\Z)^s$ be the part on which $(\Z/p\Z)^m$ acts trivially, write $V = V_{triv} \oplus W$, and let $G := W \rtimes (\Z/p\Z)^m$. Then $H$ decomposes as a direct product $H = V_{triv} \times G$, and since $V_{triv} \simeq (\Z/q\Z)^s \subseteq D_{p,q}^s$ it suffices to show that $G$ is isomorphic to a subgroup of $D_{p,q}^r$ for some $r$. Now write $W = \bigoplus_{i=1}^r W_i$ with $W_i$ irreducible representations of $(\Z/p\Z)^m$. Then every $W_i$ is nontrivial by definition of $W$. All nontrivial irreducible $\F_q$-linear representations of $(\Z/p\Z)^m$ are of the form $\rho_{p,q} \circ \pi$ for some $\pi : (\Z/p\Z)^m \to \Z/p\Z$; let $\pi_i : (\Z/p\Z)^m \to \Z/p\Z$ be the map such that $W_i$ is isomorphic to the representation $\rho_{p,q} \circ \pi_i$ and let $f_i : W_i \to \F_q(\zeta_p)$ be an $\F_q$-linear map witnessing this isomorphism. From the morphisms $f_i$ and $\pi_i$ one gets a morphism $\phi_i : G \to D_{p,q}$ by the universal property of the semidirect product.
$$\xymatrix{
W = \bigoplus_{i=1}^r W_i \ar[r]\ar[d]^{f_i}& G = (\bigoplus_{i=1}^r W_i) \rtimes (\Z/p\Z)^m  \ar@{-->}[d]^{\exists! \phi_i} & (\Z/p\Z)^m \ar[l] \ar[d]^{\pi_i}\\
\F_q(\zeta_p) \ar[r] & D_{p,q}= \F_q(\zeta_p) \rtimes \Z/p\Z & \Z/p\Z \ar[l]
} $$
Let $\phi = (\phi_1,\ldots, \phi_r) : G \to D_{p,q}$. Since $\phi|_W = (f_1,...,f_r)$ is an isomorphism one sees that $\ker \phi$ has empty intersection with $W$ and in particular this means that $\phi$ together with the quotient map $q_W$ gives an injection $(\phi,q_W) :G \hookrightarrow D_{p,q}^r \times G/W$  and the desired implication finally follows  since $G/W = (\Z/p\Z)^m \hookrightarrow D_{p,q}^m$.
\end{proof}
From the above proof it also follows that as soon as $H$ is non-commutative then $H$ has a surjective map to $D_{p,q}$, so that in this case $D_{p,q}$ is of strong $H$-type, and in particular we have the following corollary.

\begin{corollary}
	Let $H$ be a non commutative group such that $\lambda_q(\lambda_p(H)) = \set{e}$. Then the notions of weak $H$-type and weak $D_{p,q}$-type are equivalent.
\end{corollary}

\begin{remark}\label{rem:why-use-gen-type}
  This simple criterion for checking whether a given group is of weak $D_{p,q}$-type is one of the primary advantages of this theory. For instance, since $A_4 \simeq D_{3,2}$ and for $A_4\subseteq S_4$ the notions of weak and generalized $G$-type agree, one immediately has Proposition \ref{prop:A4_Classification}.
\end{remark}

\begin{lemma}\label{lem:strong_Dpq}
	Let $p,q$ be distinct primes and $H$ a group of weak $D_{p,q}$-type. Write $H \simeq (\Z/q\Z)^n \rtimes_\rho (\Z/p\Z)^m$ for some integers $n,m$ and let $\rho : (\Z/p\Z)^m \to \GL_n(\Z/q\Z)$ be a representation. Then the following are equivalent:
	\begin{enumerate}
		\item[\rm (1)] $H$ is of strong $D_{p,q}$-type.
		\item[\rm (2)] The trivial representation does not occur as a subrepresentation of $\rho$.
		\item[\rm (3)] The trivial representation does not occur as quotient representation of $\rho$.
		\item[\rm (4)] $H$ does not have a quotient isomorphic to $\Z/q\Z$.
	    \item[\rm (5)] $H/[H,H] = (\Z/p\Z)^m$.
	\end{enumerate}
\end{lemma}
Note that $pH \simeq (\Z/q\Z)^n$ and $H/pH \simeq (\Z/p\Z)^m$. Furthermore, since $pH$ is abelian, $\rho$ only depends on a choice of basis for $pH$ and $H/pH$ and not on the choice of the section $H/pH \to H$ that was used to write $H$ as a semidirect product. In particular items (2) and (3) do not depend on the way in which $H$ was written as a semidirect product.
\begin{proof}
The equivalence of (4) and (5) is trivial. The equivalence of (2) and (3) follows because  $q$ is coprime to $\# (\Z/p\Z)^m$ and hence $V$ can be written as a direct sum of irreducible representations.

The equivalence of (3) and (4) follows directly from the universal property. Indeed, let $f_1 : (\Z/p\Z)^m \to \Z/q\Z$ be a morphism. Then $f_1 = 0$, so conjugation by $f_1(x)$ on $\Z/q\Z$ is the trivial action for all $x \in (\Z/p\Z)^m$.  In particular $f_1$ can be extended to a surjective map $f : (\Z/q\Z)^n \rtimes (\Z/p\Z)^m \to \Z/q\Z$ if and only if the trivial representation is a quotient of $V$.

The implication $(2) \Rightarrow (1)$  is similar to the proof of $(4) \Rightarrow (1)$ in Proposition \ref{prop:weak_Dpq} and the same notation will be used. Since $V_{triv} = 0$ by assumption one has that the $G$ and $H$ in that proof are the same. Now since the $\pi_i$ and the $f_i$ are surjective, one sees that the $\phi_i : G \to D_{p,q}$ are surjective. In particular, $G$ is a subdirect product of $\prod_{i=1}^r D_{p,q} \times G/W$.  Thus $H=G$  from Lemma \ref{lem:easy_list} (6) because  $D_{p,q}^r$ and $G/W = (\Z/p\Z)^m$ are both of strong $D_{p,q}$-type.

Now we consider the implication $(1) \Rightarrow (2)$. Since the representation $\rho_{p,q} : \Z/p\Z \to \F_q(\zeta_p)^*$ corresponding to $D_{p,q}$ does not contain the trivial representation as a subrepresentation it suffices to show that the property of not having a trivial subrepresentation is maintained under taking quotients and subdirect products.

First we handle quotients. Let $H$ be a group of strong $D_{p,q}$-type such that the associated representation $\rho$ does not have a trivial subrepresentation and $H \to H'$ is a surjective map of groups. If $H$ does not have a quotient isomorphic to $\Z/q\Z$ then neither does $H'$ so that now the fact that the representation $\rho'$ associated to $H'$ does not have a trivial subrepresentation follows from the equivalence between $(2)$ and $(4)$.

For the invariance under subdirect products let $H_1, H_2$ be groups of strong $D_{p,q}$-type such that for $i=1,2$ the associated representations $\rho_i : H_i/pH_i \to \aut_{\F_q} (pH_i)$ do not contain a trivial subrepresentation and let $H \subset H_1 \times H_2$ be a subdirect product. Define $H_{triv} := (pH)^{H/pH}$ to be the part of $pH$ on which $H/pH$ acts trivially. Since $H \to H_i$ is surjective one has that $H/pH \to H_i/pH_i$ is surjective and hence the image of $H_{triv} \to H_i$ is contained in $(pH_i)^{H/pH} = (pH_i)^{H_i/pH_i} = 0$ so that $H_{triv}=0$.
\end{proof}

\begin{corollary}\label{cor:strong-Dpq}
A group $H$ is of strong $D_{p,q}$-type if and only if $H$ is isomorphic to a subdirectproduct of $D_{p,q}^k$ for some integer $k$.
\end{corollary}
\begin{proof}The extra fact that a strong $D_{p,q}$-type group has no quotient isomorphic to $\Z/q\Z$ shows that in the proof of the equivalence of (1) and (2) in \cref{prop:weak_Dpq} one can replace subgroup by subdirect product.
\end{proof}

\subsection{The case $G=A_4$}\label{sec:just-A4}
Using the general theory that we developed in the previous subsections, we now turn our attention to studying the infinite extension $\QA/\QQ$.

We begin by examining the relationship between $\QQ(A_4^\infty)$ and its more natural counterpart, the compositum of all $A_4$ extensions of $\QQ$, which we will denote by $\Q_{A_4}$. It is natural to ask whether $\QA$ is equal to $\Q_{A_4}$. Indeed, when $A_4$ is replaced by $D_4$, the analogous statement turns out to be true \cite[Theorem 1.11]{D4}. However, in our present setting the situation is as follows.

 First, recall that $A_4$ has no normal subgroup of index $2$. The same remains true for groups of strong $A_4$-type by Lemma \ref{lem:strong_Dpq} and the fact that $A_4 \simeq D_{3,2}$. In particular, we have the following.

 \begin{corollary}\label{cor:no_quad_subfields}
The compositum of all $A_4$ extensions $\Q_{A_4}$ contains no subfields $\Q \subseteq K \subseteq \Q_{A_4}$ such that $[K:\QQ]=2$.
 \end{corollary}

 On the other hand, the following characterization of groups of generalized $A_4$-type follows immediately from Proposition \ref{prop:weak_Dpq}; see also Remark \ref{rem:why-use-gen-type}.

\begin{prop}\label{prop:A4_Classification}
A finite group $G$ is of generalized $A_4$-type if and only if $\lambda_2(\lambda_3(G))$ is the trivial group.
\end{prop}

As an immediate consequence, we have the following corollary.

\begin{cor}\label{cor:genA4-exponent-divides-six}
If $G$ is of generalized $A_4$-type then the exponent of $G$ divides 6.
\end{cor}

We also have the following useful lemma.

\begin{lemma}\label{lem:quadratics-are-genA4}
Given any square-free $d \in \ZZ$, we have $\QQ(\sqrt{d}) \subseteq \QA$.
\end{lemma}
\begin{proof}
For any such $d$, $\mathrm{Gal}(\QQ(\sqrt{d})/\QQ) \simeq \ZZ / 2 \ZZ$ which is of generalized $A_4$-type by Proposition \ref{prop:A4_Classification}.
\end{proof}

With this result in hand, the relationship between $\QA$ and $\Q_{A_4}$ becomes clearer.

\begin{prop}\label{prop:F_ne_A4}
Let $\Q_{A_4}/\QQ$ be the compositum of all $A_4$ extensions of $\QQ$. Then $\Q_{A_4} \subsetneq \QQ(A_4^\infty)$
\end{prop}

\begin{proof}
Since the only transitive proper normal subgroup of $A_4$ is isomorphic to the Klein four-group $V_4$, the field $\QQ(A_4^\infty)$ can be viewed as the compositum of all $A_4$ and $V_4$ extensions of $\QQ$. With this we immediately get that $\Q_{A_4} \subseteq \QQ(A_4^\infty)$. To see the strict inclusion, notice that if $\Q_{A_4}$ were to contain any $V_4$ extensions of $\QQ$, then $\Q_{A_4}$ would contain a quadratic extension of $\QQ$, but this can't be the case from Corollary \ref{cor:no_quad_subfields}.
\end{proof}

In fact, we have the following description of $\QA$.

\begin{cor}
Let $\QQ(2^\infty)$ be the compositum of all quadratic extensions of $\QQ$ and let $
\Q_{A_4}$ be as in Proposition \ref{prop:F_ne_A4}. Then $\QQ(A_4^\infty)$ is the compositum of $\Q_{A_4}$ and $\QQ(2^\infty)$.
\end{cor}

For $n \geq 1$, for the rest of the paper we will write $\zeta_n$ for an $n$-th root of unity.

\begin{lemma}\label{lem:cyclotomic-genA4}
Let $n \in \NN$. Then $\QQ(\zeta_n) \subseteq \QA$ if and only if $n$ divides $504$.
\end{lemma}
\begin{proof}
Let $\lambda(n)$ be the exponent of $\mathrm{Gal}(\QQ(\zeta_n)/\QQ) \simeq \left( \ZZ / n \ZZ \right)^\times$; then $\lambda(2^e)=2^{e-2}$ for $e \geq 3$ and $\lambda(p^e)=(p-1)p^{e-1}$ for odd primes $p>2$ and $e \geq 1$. It follows that for $n$ a prime power, $\lambda(n)$ divides $6$ if and only if $n \in \{ 2,3,4,7,8,9 \}$. Thus, if $\QQ(\zeta_n) \subseteq \QA$ then $n \in \{ 2,3,4,7,8,9 \}$ by Corollary \ref{cor:genA4-exponent-divides-six}. But now it's easy to check that for each such $n$ and $G = (\ZZ / n \ZZ)^\times$ we have that $\lambda_2(\lambda_3(G))$ is trivial, and the result follows from Prop \ref{prop:A4_Classification}.
\end{proof}

\section{Growth of the torsion subgroup of elliptic curves by base extensions}\label{sec:growth-of-torsion}

Here we collect some results from other papers which will be useful in the subsequent sections. We direct the interested reader to the corresponding papers for their proofs. We begin with a result on the relationship between torsion subgroups and roots of unity.

\begin{prop}[\cite{Silverman}, Ch. III, Cor. 8.1.1]\label{prop:n-torsion-implies-nth-cyclotomic}
Let $E / L$ be an elliptic curve with $L \subseteq \bar{\QQ}$. For each integer $n \geq 1$, if $E[n] \subseteq E(L)$ then $\QQ(\zeta_n) \subseteq L$.
\end{prop}

Throughout the rest of the paper we will make extensive use of the classification of isogenies of elliptic curves defined over $\Q$. Given such a curve, we will say that it possesses an $n$-isogeny if it admits a degree $n$ isogeny with cyclic kernel.

\begin{thm}[\cite{Mazur2},\cite{Kenku1},\cite{Kenku2},\cite{Kenku3},\cite{Kenku4}]\label{thm:n-isogeny-list}
Let $E / \QQ$ be an elliptic curve with a rational $n$-isogeny. Then
\[
n \leq 19 \ \text{or} \ n \in \{ 21, 25, 27, 37, 43, 67, 163 \}.
\]
\end{thm}

\begin{lemma}[\cite{Q(3)}, Lemma 4.6]\label{lem:torsion-implies-isogeny}
Let $E$ and $F$ be as in Theorem \ref{thm:finite_over_F}, let $p$ be a prime, and let $k$ be the largest integer for which $E[p^k] \subseteq E(F)$. If $E(F)_{\tor}$ contains a subgroup isomorphic to $\ZZ / p^k \ZZ \oplus \ZZ / p^j \ZZ$ with $j \geq k$, then $E$ admits a rational $p^{j-k}$-isogeny. Moreover, $j-k \leq 4, 3, 2$ if $p=2, 3, 5$, respectively, $j-k \leq 1$ if $p=7, 11, 13, 17, 19, 43, 67, 163$, and $j=k$ otherwise.
\end{lemma}

\begin{lemma}[\cite{Q(3)}, Lemma 5.12]\label{lem:two-distinct-p-isogenies}
Let $E / \QQ$ be an elliptic curve and let $p > 2$ be a prime for which $E$ admits a rational $p$-isogeny. Then $[\QQ(E[p]):\QQ]$ is relatively prime to $p$ if and only if $E$ admits two rational $p$-isogenies (with distinct kernels). For $p>5$ the latter is impossible and hence $p \mid [\QQ(E[p]):\QQ]$.
\end{lemma}

\begin{lemma}[\cite{Q(3)}, Lemma 4.8]\label{lem:isogeny-field-cyclic}
Let $E / \QQ$ be an elliptic curve that admits a rational $n$-isogeny $\phi$, and let $R \in E[n]$ be a point of order $n$ in the kernel of $\phi$, and write $\QQ(R)=\QQ(x(R),y(R))$ for the field of definition of $R$. The field extension $\QQ(R) / \QQ$ is Galois with Galois group isomorphic to a subgroup of $(\ZZ / n \ZZ)^\times$. In particular, if $n$ is prime then $\mathrm{Gal}(\QQ(R)/\QQ)$ is cyclic and its order divides $n-1$.
\end{lemma}

\begin{prop}[\cite{L-R13}, Theorem 2.1]\label{prop:torsion-field-degree-lower-bound}
Let $E / \QQ$ be an elliptic curve and let $p \geq 11$ be a prime other than $13$. Let $R \in E[p]$ be a torsion point of exact order $p$ and let $\QQ(R)$ be as in Lemma \ref{lem:isogeny-field-cyclic}. Then
\[
[\QQ(R):\QQ] \geq \frac{p-1}{2}
\]
unless $j(E)=-7 \cdot 11^3$ and $p=37$, in which case $[\QQ(R):\QQ] \geq (p-1)/3 = 12.$
\end{prop}

Finally, we observe that torsion structures for $E(\QA)$ are almost entirely determined by the $j$-invariant $j(E)$.

\begin{prop}\label{prop:j-invariant}
Let $E / \QQ$ be an elliptic curve with $j(E) \neq 0,1728$. The isomorphism type of $E(\QA)_{\mathrm{tors}}$ depends only on the $\bar{\QQ}$-isomorphism class of $E$, equivalently, on $j(E)$.
\end{prop}
\begin{proof}
Recall that for $j(E)\neq 0$ or $1728$, if $j(E)=j(E')$ for some $E'/\QQ$ then $E'$ is a quadratic twist of $E$, hence isomorphic to $E$ over an extension of $\QQ$ of degree at most $2$, and quadratic fields are of generalized $A_4$-type.
\end{proof}

Putting these results together, we obtain a bound on the primes which may divide $\#E(\QA)_{\mathrm{tors}}$.

\begin{prop}\label{prop:eligible-p-torsion}
Let $E / \QQ$ be an elliptic curve, and let $p$ be a prime dividing the cardinality of $E(\QA)_{\tor}$. Then $p \in \{ 2,3,5,7,13 \}$.
\end{prop}
\begin{proof}
First, by Proposition \ref{prop:n-torsion-implies-nth-cyclotomic} and Lemma \ref{lem:torsion-implies-isogeny}, the nontrivialtiy of the $p$-primary component $E(\QA)(p)$ of the torsion subgroup of $E(\QA)$ implies that either $\QQ(\zeta_p) \subseteq \QA$ or that $E$ has a rational $p$-isogeny. Thus, by Lemma \ref{lem:cyclotomic-genA4} and Theorem \ref{thm:n-isogeny-list}, the only primes which may divide the cardinality of $E(\QA)_{\mathrm{tors}}$ are the primes in the set
\[
S=\{ 2,3,5,7,11,13,17,19,37,43, 67, 163 \}.
\]
By Lemma \ref{lem:isogeny-field-cyclic}, if $E(\QA)[p] = \langle R \rangle \simeq \ZZ / p \ZZ$, then $\QQ(R)/\QQ$ is a cyclic extension, and it must also be of generalized $A_4$-type, hence it is an extension of degree dividing $6$. Combining this with Proposition \ref{prop:torsion-field-degree-lower-bound} completes the proof.
\end{proof}

\section{The maximal $p$-primary components of $E(\QA)_{\tor}$}\label{sec:pPrimary}

In this section we obtain bounds on the $p$-primary components of $E(\QA)_{\tor}$ for elliptic curves $E$ defined over $\QQ$. Taken together, the results of this section establish the following theorem.

\begin{thm}\label{thm:non-CM-torsion-bound}
Let $E / \QQ$ be an elliptic curve. Then $E(\QA)_{\tor}$ is isomorphic to a subgroup of
\[
T_{\text{max}} = (\ZZ/8\ZZ\oplus\ZZ/16\ZZ) \oplus (\ZZ / 3 \ZZ \oplus \ZZ / 9 \ZZ) \oplus \ZZ / 5\ZZ \oplus \ZZ / 7 \ZZ \oplus \ZZ / 13 \ZZ
\]
and $T_{\text{max}}$ is the smallest group with this property.
\end{thm}

\subsection{Primes without the possibility of full $p$-torsion ($p=5,13$)}

By Lemma \ref{lem:cyclotomic-genA4} and Proposition \ref{prop:n-torsion-implies-nth-cyclotomic}, we see that if $p \in \{ 5, 13\}$ then the full $p$-torsion $E[p]$ cannot be defined over $\QA$. We now study each of these cases in detail.

\subsubsection{The case when $p=13$}

\begin{prop}\label{prop:case-p-13}
Suppose that $E / \QQ$ is an elliptic curve such that $13$ divides $\# E(\QA)_{\mathrm{tors}}$. Then $E(\QA)(13) \simeq \ZZ / 13 \ZZ$ and there exists $t \in \QQ$ such that
\begin{equation}\label{eq:j13}
j(E) = \frac{(t^4 - t^3 + 5t^2 + t + 1)P(t)^3}{t^{13}(t^2-3t-1)}
\end{equation}
where $P(t)=t^8 - 5t^7 + 7t^6 - 5t^5 + 5t^3 + 7t^2 + 5t + 1$.
\end{prop}
\begin{proof}
Since $\QQ(\zeta_{13}) \not\subseteq \QA)$ by Lemma \ref{lem:cyclotomic-genA4}, Proposition \ref{prop:n-torsion-implies-nth-cyclotomic} implies that $E(\QA)[13] \simeq \ZZ / 13 \ZZ$. Suppose that $R$ is a point that generates $E(\QA)[13]$. By Lemmas \ref{lem:torsion-implies-isogeny} and \ref{lem:isogeny-field-cyclic}, we see that $\QQ(R)/ \QQ$ must be a cyclic extension of degree dividing $12$. Furthermore, since $R$ is defined over $\QA$, Corollary \ref{cor:genA4-exponent-divides-six} implies that this cyclic extension must have degree dividing $6$. Thus, identifying $E[13]$ with column vectors, the image of $\bar{\rho}_{E,13}$ is conjugate to a subgroup of $\mathrm{GL}_2(\ZZ / 13 \ZZ)$ contained inside the matrices of the form $\left( \begin{smallmatrix} a^2 & \ast \\ 0 & \ast  \end{smallmatrix} \right)$.

The elliptic curves defined over $\QQ$ with this property have been completely classified in \cite{Zywina}, and they correspond to the curves with $j$-invariant of the form given in \eqref{eq:j13}. Further, since there are no elliptic curves defined over $\QQ$ with a cyclic $169$-isogeny, it is not possible for the $13$-primary component of $E(\QA)$ to be any larger.
\end{proof}

\subsubsection{The case when $p=5$}

\begin{prop}\label{prop:case-p-5}
Suppose that $E / \QQ$ is an elliptic curve such that $5$ divides $\# E (\QA)_{\tor}$. Then $E(\QA)(5) \simeq \ZZ / 5 \ZZ$ and there exists $t \in \QQ$ such that
\begin{equation}\label{eq:j5}
j(E)= \frac{(t^4 - 12t^3 + 14t^2 + 12t + 1)^3}{t^5 (t^2-11t-1)}.
\end{equation}
\end{prop}
\begin{proof}
As in the previous proposition, since $\QQ(\zeta_5) \not\subseteq \QA$ by Lemma \ref{lem:cyclotomic-genA4}, it follows from Proposition \ref{prop:n-torsion-implies-nth-cyclotomic} that $E(\QA)[5]$ is cyclic of order $5^j$ for some $j$, and by Lemma \ref{lem:torsion-implies-isogeny} and Theorem \ref{thm:n-isogeny-list} we must have $j \leq 2$.

Suppose $j=1$, so $E(\QA)(5)=\ZZ / 5\ZZ$, and let $R$ be a point of order $5$. Then by Lemma \ref{lem:torsion-implies-isogeny}, $E$ admits a rational $5$-isogeny, and by Lemma \ref{lem:isogeny-field-cyclic} its field of definition $\QQ(R)$ is an extension of $\QQ$ of degree at most $4$. It follows from Corollary \ref{cor:genA4-exponent-divides-six} that $R$ is defined over a quadratic extension of $\QQ$, or in other words, there is a quadratic twist of $E$ that has a rational point of order $5$. Such curves are parametrized by the genus $0$ modular curve $X_1(5)$, and the $j$-map for this curve is given in \cite{Zywina} as \eqref{eq:j5}.

Now suppose $E(\QA)$ contains a point $R$ of order $25$; then $E$ admits a rational $25$-isogeny, and by Lemma \ref{lem:isogeny-field-cyclic} the field of definition $\QQ(R)$ is Galois with Galois group isomorphic to a subgroup of $(\ZZ / 25 \ZZ)^\times$, which has order $20$. But since $\QQ(R) \subseteq \QA$, by Corollary \ref{cor:genA4-exponent-divides-six} we must have that $\QQ(R)/\QQ$ is a field extension of degree dividing $6$. These facts together imply $R$ is defined over a quadratic extension of $\QQ$. If $E$ has a point of order 25 defined over a quadratic extension of $\QQ$, then there is a quadratic twist of $E$ with a rational point of order 25, but this cannot happen by Theorem \ref{thm:mazur}.
\end{proof}

\subsection{Primes with the possibility of full $p$-torsion ($p=2,3,7$)}

We now consider the primes $p$ for which the full $p$-torsion $E[p]$ may be defined over $\QA$.

\subsubsection{The case when $p=7$}




\begin{prop}\label{prop:case-p-7}
Suppose that $E / \QQ$ is an elliptic curve such that $7$ divides $\# E (\QA)_{\tor}$. Then $E(\QA)_{\mathrm{tors}} \simeq \ZZ / 7 \ZZ$ and there exists some $t \in \QQ^\times$ such that
\begin{equation}\label{eq:j7}
j(E)= \frac{(t^2 + 13t + 49)(t^2 + 5t + 1)^3}{t}.
\end{equation}
\end{prop}

\begin{proof}
The proof is essentially the same as that of \cite[Lemma 5.13]{Q(3)}, replacing $\QQ(3^\infty)$ by $\QA$. We begin by noting that Lemma \ref{lem:cyclotomic-genA4} and Proposition \ref{prop:n-torsion-implies-nth-cyclotomic} imply that $E[49] \not\subseteq \QA$, so by Lemma \ref{lem:torsion-implies-isogeny} we have $E(\QA) \simeq \ZZ / 7^k\ZZ \oplus \ZZ / 7^j \ZZ$ with $k \leq 1$ and $k \leq j \leq k+1$.

If $j > k$ then Lemma \ref{lem:torsion-implies-isogeny} implies that $E$ admits a rational $7$-isogeny, and then Lemma \ref{lem:two-distinct-p-isogenies} implies that $[\QQ(E[7]):\QQ]$ is divisible by $7$. But then the exponent of $\mathrm{Gal}(\QQ(E[7])/\QQ)$ does not divide $6$, so by Corollary \ref{cor:genA4-exponent-divides-six} we see that $\QQ(E[7]) \not\subseteq \QA$. Therefore $k=0, j=1$, and $E(\QA) \simeq \ZZ / 7 \ZZ$. This also rules out the case $k=1$ and $j=2$ proving the first statement in the theorem.

If $j=k$, then $E$ cannot admit a rational $7$-isogeny; to see this, assume a rational $7$-isogeny does exist, and let $R$ be a non-trivial point in its kernel. Then by Lemma \ref{lem:isogeny-field-cyclic}, $\mathrm{Gal}(\QQ(R)/ \QQ)$ is cyclic of order dividing $6$, hence $\QQ(R) \subseteq \QA$. But then we must have $j=k=1$, so $\QQ(E[7]) \subseteq \QA$, and so by Lemma \ref{lem:two-distinct-p-isogenies} the degree $[\QQ(E[7]):\QQ]$ is divisible by $7$, contradicting the fact that $\QQ(E[7]) \subseteq \QA$. Thus $k=0$ and $j=1$ if and only if $E$ admits a rational $7$-isogeny, and such elliptic curves correspond to points on the modular curve $X_0(7)$, whose $j$-map is given in \cite[Table 3]{L-R13} as \eqref{eq:j7}.

Now suppose $j=k=1$. Then $\QQ(E[7]) \subseteq \QA$, so $\mathrm{Gal}(\QQ(E[7])/\QQ)$ has exponent dividing $6$ by Corollary \ref{cor:genA4-exponent-divides-six}. This implies that for every prime $p \neq 7$ of good reduction for $E$, the elliptic curve $E_p / \FF_p$ obtained by reducing $E$ modulo $p$ has its $7$-torsion defined over an $\FF_p$-extension of degree dividing $6$, and in particular, admits an $\FF_p$-rational $7$-isogeny. Thus $E / \QQ$ admits a rational $7$-isogeny locally everywhere but not globally, and as proved in \cite{Sutherland12}, this implies $j(E)=2268945/128$. It is also proved in {\it loc.~cit.}~ that, conversely, for every elliptic curve $E / \QQ$ with this $j$-invariant the group $\mathrm{Gal}(\QQ(E[7])/\QQ)$ is isomorphic to a subgroup of $\mathrm{GL}_2(\FF_7)$ with surjective determinant map whose image in $\mathrm{PGL}_2(\FF_7)$ is isomorphic to $S_3$; up to conjugacy there are exactly two such groups (labeled 7NS.2.1 and 7NS.3.1 in \cite{Sutherland16}). However, one checks using Magma that neither of these groups is of generalized $A_4$-type, and so the case $j=k=1$ is impossible.
\end{proof}

\subsubsection{The case when $p=3$}

\begin{lemma}\label{lem:full-3-torsion}
Let $E / \QQ$ be an elliptic curve. Then the 3-primary component $E(\QA)(3)$ is nontrivial if and only if $E$ admits a rational $3$-isogeny. Furthermore, if $E[3]$ is defined over $\QA$, then $E$ admits two distinct rational $3$-isogenies.
\end{lemma}
\begin{proof}
Suppose that $E(\QA)(3)$ is nontrivial. If $\QQ(E[3]) \not \subseteq \QA$, then by Lemma \ref{lem:torsion-implies-isogeny} $E(\QA)(3)$ can only be nontrivial if $E$ admits a rational $3$-isogeny. On the other hand, suppose $E(\QA)[3]=E[3]$, or equivalently $\QQ(E[3]) \subseteq \QA$. Since $\mathrm{Gal}(\QQ(E[3])/\QQ)\simeq \Im \bar{\rho}_{E,3}$, this implies $\Im \bar{\rho}_{E,3} \subseteq \mathrm{GL}_2(\ZZ / 3 \ZZ)$ must
\begin{enumerate}[(i)]
\item be of generalized $A_4$-type,
\item have surjective determinant map, and
\item contain an element of trace $0$ and determinant $-1$.
\end{enumerate}
The first condition comes from the assumption that $E(\QA)[3]=E[3]$, while conditions (ii) and (iii) are always satisfied by the image of a Galois representation coming from an elliptic curve. Enumerating these groups, one finds that the addition of conditions (ii) and (iii) forces $G$ to be conjugate to a subgroup of the split Cartan subgroup, proving the last statement of the lemma.
\end{proof}

\begin{lemma}\label{lem:full-9-torsion}
Let $E / \QQ$ be an elliptic curve. Then $E[9] \not\subseteq E(\QA)$.
\end{lemma}
\begin{proof}
Suppose that $E(\QA)[9]=E[9]$; then the subgroup  $G=\Im \bar{\rho}_{E,9} \subseteq \mathrm{GL}_2(\ZZ /9 \ZZ)$ must satisfy the following conditions:
\begin{enumerate}[(i)]
\item $G$ has a surjective determinant map and an element with trace $0$ and determinant $-1$, and
\item $G$ is of generalized $A_4$-type.
\end{enumerate}
Enumerating these groups, one finds that, up to conjugation, there is exactly one maximal subgroup $H \subseteq \mathrm{GL}_2(\ZZ / 9 \ZZ)$ with this property, and it has the following generating set:
\[
H = \left\langle  \left( \begin{matrix} 1 & 0 \\ 0 & 8 \end{matrix} \right), \left( \begin{matrix} 1 & 0 \\ 0 & 4 \end{matrix} \right), \left( \begin{matrix} 8 & 0 \\ 0 & 8 \end{matrix} \right), \left( \begin{matrix} 7 & 0 \\ 0 & 4 \end{matrix} \right) \right\rangle.
\]
Thus $H$ is contained in the split Cartan subgroup, and so $E$ has two distinct rational $9$-isogenies. We thus have elliptic curves $E_1$ and $E_2$, both defined over $\QQ$, and the isogeny graph of $E$ contains a subgraph of the form
\[
E_1 \stackrel{9}{\longleftrightarrow} E \stackrel{9}{\longleftrightarrow} E_2
\]
where the number over the arrow indicates the degree of the isogeny. But this implies the existence of a rational $81$-isogeny between $E_1$ and $E_2$, and this is impossible by Theorem \ref{thm:n-isogeny-list}.
\end{proof}

\begin{lemma}\label{lem:3x9}
Suppose that $E / \QQ$ is an elliptic curve such that $3$ divides $\# E(\QA)_\tor$. Then $E(\QA)(3)$ is isomorphic to a subgroup of
\[
\ZZ / 3 \ZZ \oplus \ZZ / 9 \ZZ.
\]
\end{lemma}

\begin{proof}
By Lemma \ref{lem:full-3-torsion}, if $E(\QA)(3)$ contains a subgroup isomorphic to $\ZZ / 3 \ZZ \oplus \ZZ / 3 \ZZ$, then $E$ must admit two distinct rational $3$-isogenies. If $E(\QA)(3)$ contained a subgroup isomorphic to $\ZZ / 3 \ZZ \oplus \ZZ / 27 \ZZ$, then by Lemmas \ref{lem:torsion-implies-isogeny} and \ref{lem:full-3-torsion} we see that $E$ would admit a $3$- and a $9$-isogeny where the kernel of the 3-isogeny is not contained in the kernel of the 9-isogeny (by abuse of terminology we call such isogenies \textit{independent}). So in fact $E$ would be isogenous to a curve which admits a rational $27$-isogeny. Since $X_0(27)$ is a rank 0 genus 1 curve, there are only finitely many $\Qbar$-isomorphism classes of elliptic curves that are 3-isogenous to an elliptic curve with a 27-isogeny. Checking the field of definition of these points of order 27, we see that none of them have a point of order 27 defined over $\QA$. Thus, in light of Lemma \ref{lem:full-9-torsion}, $E(\QA)(3)$ is indeed a subgroup of $\ZZ / 3 \ZZ \oplus \ZZ / 9 \ZZ.$
\end{proof}

The next lemma shows that this bound is sharp. See also Proposition \ref{prop:CMj0}.

\begin{lemma}\label{lem:z3-z9-example} If $E(\QA)(3) \simeq \ZZ / 3 \ZZ \oplus \ZZ / 9 \ZZ$, then $E$ is isogenous to the elliptic curve with Cremona label \href{http://www.lmfdb.org/EllipticCurve/Q/54b3}{\tt 27a1}. \end{lemma}

  \begin{proof}
    Suppose that $E(\QA)(3) \simeq \ZZ / 3 \ZZ \oplus \ZZ / 9 \ZZ$; then the subgroup  $G=\Im \bar{\rho}_{E,9} \subseteq \mathrm{GL}_2(\ZZ /9 \ZZ)$ must satisfy the following conditions:
    \begin{enumerate}[(i)]
    \item $G$ has a surjective determinant map and an element with trace $0$ and determinant $-1$, and
    \item $G$ contains a normal subgroup $N$ that acts trivially on a $\ZZ / 9 \ZZ$-submodule of $\ZZ / 9 \ZZ \oplus \ZZ / 9 \ZZ$ isomorphic to $\ZZ / 3 \ZZ \oplus \ZZ / 9 \ZZ$ for which $G / N$ is of generalized $A_4$-type.
    \end{enumerate}
    Enumerating these groups, one finds that, up to conjugation, there is exactly one maximal subgroup $H \subseteq \mathrm{GL}_2(\ZZ / 9 \ZZ)$ with this property, and it has the following generating set:
    \[
    H = \left\langle  \left( \begin{matrix} 1 & 0 \\ 0 & 8 \end{matrix} \right), \left( \begin{matrix} 1 & 0 \\ 0 & 4 \end{matrix} \right), \left( \begin{matrix} 8 & 0 \\ 0 & 8 \end{matrix} \right), \left( \begin{matrix} 7 & 0 \\ 0 & 4 \end{matrix}\right), \left( \begin{matrix} 4 & 3 \\ 0 & 4 \end{matrix} \right) \right\rangle.
    \]
  Thus $E$ admits independent $3$- and $9$-isogenies, so in fact it is isogenous to a curve $E'$ which admits a $27$-isogeny. From \cite[Table 3]{L-R13} we see $j(E') = -12288000$, and a simple computation shows that $j(E) = 0$. Magma confirms that if $E$ is the elliptic curve with Cremona label \href{http://www.lmfdb.org/EllipticCurve/Q/54b3}{\tt 27a1}, then $E(\QA)(3) \simeq \ZZ / 3 \ZZ \oplus \ZZ / 9 \ZZ$.
  \end{proof}

\begin{remark}\label{rmk:z9-example}
If $E$ is the curve with Cremona label \href{http://www.lmfdb.org/EllipticCurve/Q/54b3}{\tt 54b3}, then $E(\QA)_\tor \simeq \ZZ / 9 \ZZ$.
\end{remark}

\subsubsection{The case when $p=2$}

The last case we must is when $p=2$. In this section alone, in order to use the results of \cite{RZB}, we will make the assumption that $E$ does not have complex multiplication; since we will deal with CM elliptic curves in a separate section (see Section \ref{sec:CM}), this assumption will not ultimately impede our progress.

Recall that by Lemma \ref{lem:cyclotomic-genA4} and Proposition \ref{prop:n-torsion-implies-nth-cyclotomic}, the largest full $2$-power torsion that can be contained in $E(\QA)$ is $E[8]$. Thus, by Theorem \ref{thm:n-isogeny-list} and Lemma \ref{lem:torsion-implies-isogeny}, $E(\QA)[2]$ may be as big as $\ZZ / 8 \ZZ \oplus \ZZ / 128 \ZZ$. To determine the torsion structures that actually appear, we make use of the results of Rouse and Zureick-Brown \cite{RZB}. Using Magma, we are able to search through their data and classify all $2$-torsion structures that can occur over $\QA$, yielding the following proposition.

\begin{prop}\label{prop:2-tors}
Let $E/\QQ$ be an elliptic curve such that $2$ divides $\# E(\QA)_{\mathrm{tors}}$. Then $E(\QA)(2)$ is isomorphic to one of the following 8 groups:

\begin{center}
$\{\mathcal{O}\},\ \  \ZZ/2\ZZ\oplus\ZZ/2\ZZ,\ \  \ZZ/2\ZZ\oplus\ZZ/4\ZZ,\ \  \ZZ/2\ZZ\oplus\ZZ/8\ZZ,\ \  \ZZ/2\ZZ\oplus\ZZ/16\ZZ,$\\
$ \ZZ/4\ZZ\oplus\ZZ/4\ZZ,\ \  \ZZ/4\ZZ\oplus\ZZ/8\ZZ,\ \  \ZZ/4\ZZ\oplus\ZZ/16\ZZ, \ \  \ZZ/8\ZZ\oplus\ZZ/8\ZZ$
\end{center}

\begin{corollary}\label{cor:RZBtable}
Let $E/\QQ$ be an elliptic curve. Using the notation in \cite{RZB}, for each $T$ listed in Proposition \ref{prop:2-tors} the group $E(\QA)$ contains a subgroup isomorphic to $T$ if and only if $E$ corresponds to a rational point on the modular curve given in Table \ref{table:2-primary}.

\end{corollary}
\begin{center}
\begin{table}
\begin{tabular}{|c|c||c|c|}\hline
T  & \hbox{\rm Modular Curve}  & T  & \hbox{\rm Modular Curve}    \\\hline\hline
$\ZZ/2\ZZ\oplus\ZZ/2\ZZ$  &\href{http://users.wfu.edu/rouseja/2adic/X6.html}{$X_6$} &
  $\ZZ/4\ZZ\oplus\ZZ/4\ZZ$  & \href{http://users.wfu.edu/rouseja/2adic/X2.html}{$X_{2}$}, \href{http://users.wfu.edu/rouseja/2adic/X27.html}{$X_{27}$}  \\\hline
$\ZZ/2\ZZ\oplus\ZZ/4\ZZ$  & \href{http://users.wfu.edu/rouseja/2adic/X13.html}{$X_{13}$} &
$\ZZ/4\ZZ\oplus\ZZ/8\ZZ$  & \href{http://users.wfu.edu/rouseja/2adic/X25.html}{$X_{25}$}, \href{http://users.wfu.edu/rouseja/2adic/X92.html}{$X_{92}$} \\\hline
$\ZZ/2\ZZ\oplus\ZZ/8\ZZ$  & \href{http://users.wfu.edu/rouseja/2adic/X36.html}{$X_{36}$} &
 $\ZZ / 4 \ZZ \oplus \ZZ / 16 \ZZ$ & \href{http://users.wfu.edu/rouseja/2adic/X193.html}{$X_{193}$} \\\hline
$\ZZ/2\ZZ\oplus\ZZ/16\ZZ$  & \href{http://users.wfu.edu/rouseja/2adic/X235.html}{$X_{235}$} &
 $\ZZ/8\ZZ\oplus\ZZ/8\ZZ$ & \href{http://users.wfu.edu/rouseja/2adic/X58.html}{$X_{58}$} \\\hline
\end{tabular}
\caption{Parameterizations of the possible nontrivial $2$-primary components}
\label{table:2-primary}
\end{table}
\end{center}

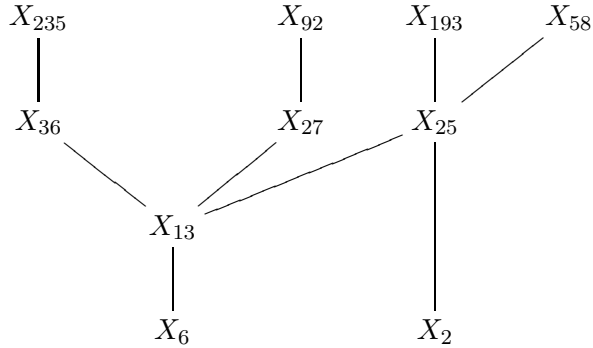
\begin{figure}
$$\xymatrix{
X_{235} \ar@{-}[d]     &   & X_{92}\ar@{-}[d]  & X_{193}\ar@{-}[d]  & X_{58}\ar@{-}[ld]  \\
  X_{36}\ar@{-}[rd] &   &X_{27}\ar@{-}[ld]   & X_{25}\ar@{-}[dd]\ar@{-}[lld]  &   \\
   &  X_{13}\ar@{-}[d] &   &   &   \\
   &  X_6 &   & X_2  &   \\
}$$
\caption{Covering relationships between the curves in Corollary \ref{cor:RZBtable}}\label{fig:RZB}
\end{figure}

\end{prop}

\section{Determining the possible torsion structures}\label{sec:Torsion}
 With Proposition \ref{prop:A4_Classification} in hand, we must now determine which subgroups of $T_\mathrm{max}$ actually occur as torsion subgroups of elliptic curves over $\QA$. The main task is to determine which combinations of the possible $p$-primary components are realized by elliptic curves $E / \QQ$ upon base extension to $\QA$.

 \subsection{The case when $E$ has complex multiplication}\label{sec:CM}

 If $E / \QQ$ is an elliptic curve with complex multiplication by an imaginary quadratic field of discriminant $D$, then $D$ belongs to the set $\{ -3,-4, -7, -8, -11, -19, -43, -67, -163 \}$; representative curves for each of these discriminants, along with their $j$-invariants, are given in \cite[Appendix A, $\S$ 3]{Silverman2}.  By Proposition \ref{prop:j-invariant}, if $E / \QQ$ has $j \neq 0,1728$, then the isomorphism type of $E(\QA)_\mathrm{tors}$ depends only on $j$. Thus, apart from $j=0,1728$, the torsion structures associated to curves with complex multiplication can be computed directly in Magma, and they are given (organized by Cremona label) in the Table \ref{table:CM-not-special}.
 \begin{center}
   \begin{table}
 \begin{tabular}{|c|c||c|c|}\hline
 $E/ \QQ$  & $E(\QA)_{\mathrm{tors}}$  & $E/ \QQ$  & $E(\QA)_{\mathrm{tors}}$    \\\hline\hline
 \href{http://www.lmfdb.org/EllipticCurve/Q/27a4}{\tt 27a4}  & $\ZZ/9\ZZ$ &
 \href{http://www.lmfdb.org/EllipticCurve/Q/361a1}{\tt 361a1}  & $\{ \mathcal{O} \}$ \\\hline
 \href{http://www.lmfdb.org/EllipticCurve/Q/32a4}{\tt 32a4} &$\ZZ/2\ZZ \oplus \ZZ / 4\ZZ$  &
 \href{http://www.lmfdb.org/EllipticCurve/Q/784h2}{\tt 784h2}  & $\ZZ / 2\ZZ \oplus \ZZ / 14 \ZZ$ \\\hline
 \href{http://www.lmfdb.org/EllipticCurve/Q/36a2}{\tt 36a2}  & $\ZZ / 2\ZZ \oplus \ZZ / 6 \ZZ$ &\href{http://www.lmfdb.org/EllipticCurve/Q/1849a1}{\tt 1849a1}  & $\{\mathcal{O}\}$ \\\hline
 \href{http://www.lmfdb.org/EllipticCurve/Q/49a1}{\tt 49a1}  & $\ZZ / 2\ZZ \oplus \ZZ / 14\ZZ$ & \href{http://www.lmfdb.org/EllipticCurve/Q/4489a1}{\tt 4489a1} & $\{ \mathcal{O} \}$ \\\hline
 \href{http://www.lmfdb.org/EllipticCurve/Q/121b1}{\tt 121b1}  & $\{\mathcal{O}\}$ &
 \href{http://www.lmfdb.org/EllipticCurve/Q/26569a1}{\tt 26569a1} & $\{ \mathcal{O}\}$ \\\hline
 \href{http://www.lmfdb.org/EllipticCurve/Q/256a1}{\tt 256a1}  & $\ZZ / 2\ZZ \oplus \ZZ / 2\ZZ$ &  &  \\\hline
 \end{tabular}
 \caption{\label{table:CM-not-special} Torsion structures of CM curves with $j \neq 0,1728$}
 \end{table}
 \end{center}

 \subsubsection{CM curves with $j=0$}\label{sec:j0} Let $E / \QQ$ be an elliptic curve with $j(E)=0$. Let us first suppose that $p \in \{5, 7, 13\}$. By considering \cite[Table 3]{L-R13}, we see that no curve with $j(E)=0$ has a $p$-isogeny, so if $p \mid \# E(\QA)_{\mathrm{tors}}$, then $E[p] \subseteq E(\QA)$, and the only possibility is $p=7$. But by Proposition \ref{prop:case-p-7}, $E[7] \subseteq E(\QA)$ implies $j(E) \neq 0$, which is a contradiction. Thus, for a curve with $j(E)=0$, we need only consider the primes $p=2$ or $3$.

 Now, every elliptic curve $E/\QQ$ with $j(E)=0$ is isomorphic over $\QQ$ to a curve of the form
 \[
 E_s \colon y^2 = x^3 + s
 \]
 for some $s \in \ZZ \setminus \{0\}$ that is $6$th-power free. Our analysis will proceed as in the proof of \cite[Lemma 5.5]{D4}. The $3$-division polynomial of $E_s$ is given by $f_3(x)=3x(x^3+4s)$, so generically such a curve has a point of order $3$ defined over a quadratic field and hence over $\QA$. It follows that such a curve always has a rational $3$-isogeny.

 Now, if the factor $x^3+4s$ is irreducible, then the full $3$-torsion will be defined over a field contained in an $S_3$ extension, which is not of generalized $A_4$-type. On the other hand, if $4s=t^3$ for some $t\in \QQ$, then we have $E[3] \subseteq E(\QA)$ since this would mean its full $3$-torsion is defined over a $2$-elementary extension of $\QQ$. In this case $E$ has two separate $3$-isogenies, and so $\Im \bar{\rho}_{E,3}$ is conjugate to a subgroup of the split Cartan subgroup of $\GL_2(\ZZ / 3 \ZZ)$. In fact, if $4s$ is a cube, the factorization of the $9$-division polynomial of $E_s$ shows that $E_s$ has a $3$-isogeny and a $9$-isogeny that are independent of each other. Notice also that if $ t,r \in \ZZ \setminus \{0\}$, then the curves $E_{2t^3}$ and $E_{2r^3}$ are isomorphic over $\QQ(\sqrt{rt}) \subseteq \QA$. Therefore, all of these curves have the same torsion subgroup over $\QA$ as $E_2$, and using Magma we find that
 \[
 E_2(\QA)_{\mathrm{tors}} = \ZZ / 3\ZZ \oplus \ZZ / 9 \ZZ.
 \]

 As noted before, the only other isogeny type that $E_s$ can have is a $2$-isogeny, which occurs precisely when $s=t^3$ for some $t \in \ZZ$ and thus $E_s$ has a point of order $2$ defined over $\QQ$. As before, if $t,r \in \ZZ \setminus \{0\}$, then $E_{t^3}$ is isomorphic to $E_{r^3}$ over $\QQ(\sqrt{rt}) \subseteq \QA$, and so $E_{t^3}(\QA)_{\mathrm{tors}} \simeq E_{r^3}(\QA)_{\mathrm{tors}}$. Thus for every $r \in \ZZ \setminus \{0\}$, we have
 \[
 E_{r^3}(\QA)_{\mathrm{tors}} \simeq E_1(\QA)_{\mathrm{tors}} \simeq \ZZ / 2 \ZZ \oplus \ZZ / 6 \ZZ
 \]
 where the second isomorphism was computed using Magma.

 Putting this all together, we obtain the following result.

 \begin{prop}\label{prop:CMj0}
 There are three possible torsion structures over $\QA$ for an elliptic curve $E / \QQ$ with $j(E)=0$. These structures are
 \[
 \ZZ / 3 \ZZ \oplus \ZZ / 9 \ZZ, \quad \ZZ / 2 \ZZ \oplus \ZZ / 6 \ZZ, \quad \text{and} \quad \ZZ / 3\ZZ,
 \]
 and they are realized by the curves \href{http://www.lmfdb.org/EllipticCurve/Q/27a1}{\tt 27a1}, \href{http://www.lmfdb.org/EllipticCurve/Q/36a1}{\tt 36a1}, and \href{http://www.lmfdb.org/EllipticCurve/Q/108a1}{\tt 108a1}, respectively.
 \end{prop}

 \subsubsection{CM curves with $j=1728$}\label{sec:j1728}

 Now let $E / \QQ$ be an elliptic curve with $j(E)=1728$. Then $E$ is isomorphic over $\QQ$ to
 \[
 E_s \colon y^2 = x^3 + sx
 \]
 where $s \in \ZZ \setminus \{ 0 \}$ is $4$th-power free. Let us first study the possible $2$-torsion for this curve.

 We see that $E_s(\QQ)$ contains the $2$-torsion point $(0,0)$, and another $2$-torsion point is defined over the quadratic field $\QQ(\sqrt{-s}) \subseteq \QA$, so we always have $E_s[2] \subseteq E_s(\QA)$.

 In fact, by considering the $4$-division polynomial of $E_s$, one verifies with Magma that $\Q(E_s[4])=\Q(i,\sqrt[4]{s},\sqrt{2})$, yielding the following lemma.

 \begin{lemma}\label{lem:j1728-full-4-torsion}
 If $E_s/\QQ$ is an elliptic curve with $j(E)=1728$, then $E_s[4] \subseteq E_s(\QA)$ if and only if  $\pm s$ is a nonzero rational square.
 \end{lemma}

 Otherwise, if $\pm s$ is not a rational square, $E_s(\QA)$ cannot have a point of order 4. To see this, we start by noticing that if $E_s(\QA)$ did have a point of order 4, then the 2-primary component of $E_s(\QA)$ would be isomorphic to $\Z/2\Z\oplus\Z/4\Z$ by the previous lemma together with the observation that $E_s(\Q)[2]$ is nontrivial and so $E_s[2]\subseteq E_s(\QA)$.
 Searching $\GL_2(\Z/4\Z)$ in Magma, we see that the only way that this is possible would be for the point of order 4, call it $P$, to be in the kernel of a 4-isogeny, with $2P \in E_s(\Q)$ and $x(P)\in\Q$. Since we are assuming $\pm s$ is not a square, we know that the only point of order 2 on $E$ defined over $\Q$ is the point $(0,0)$. Letting $P = (\alpha,\beta)$ and using the duplication formula for $E_s$ we have that
 \[
 x(2P) = \frac{\alpha^4 - 2s\alpha^2+s^2}{2\beta^2} = 0 \Longleftrightarrow  0 = \alpha^4 - 2s\alpha^2+s^2 = (\alpha^2 - s )^2.
 \]
 Clearly the only way that $\alpha$ can be in $\Q$ is if $s$ is a square, contradicting our assumption. Therefore, $E_s$ cannot have a point of order 4 defined over $\QA$ in this case. This gives the following lemma.
 \begin{lemma}\label{lem:j1728-s-nonsquare-no-point-order-4}
 If $\pm s$ is not a square in $\QQ$, then $E_s(\QA)(2) \simeq \Z/2\Z\oplus\Z/2\Z.$
 \end{lemma}

 The last thing that we have to show is that if $\pm s$ is a rational square, then $E(\QA)$ does not contain a point of order 8. Using Magma to analyze the subgroups of $\GL_2(\Z/8\Z)$, we see that in order for $E_s(\QA)$ to have a subgroup isomorphic to $\Z/4\Z \times \Z/8\Z$, $E_s$ would either have to possess an 8-isogeny or independent $2$- and 4-isogenies. A quick check shows that $j(E_s) = 1728$ is not in the image of the map $j:X_0(8)(\Q) \to \Q$, and thus $E_s$ cannot have an 8-isogeny. Further, in order for $E_s$ to have independent 2- and 4-isogenies it must be that $E[2]\subseteq E(\Q)$, which implies that $s = -t^2$ for some $t \in \Z.$ In this case we get that the 4-division polynomial is exactly
 \[
 2(x^2+t^2)(x^2-2tx-t^2)(x^2+2tx-t^2).
 \]
 In order for $E_s$ to have a 4-isogeny, this polynomial would need to have a root in $\Q$. Simple inspection shows that this could only happen if $t= 0$ or if $2$ were a square in $\Q$, neither of which is the case. Thus we have shown the following proposition.

 \begin{prop}
 Let $E_s\colon y^2 = x^3 + sx$. Then
 \[ E_s(\QA)(2) \simeq
 \begin{cases}
 \Z/4\Z \oplus \Z/4\Z & \hbox{if } \pm s \hbox{ is a rational square,}\\
 \Z/2\Z \oplus \Z/2\Z & \hbox{otherwise.}\\
 \end{cases}
 \]

 \end{prop}

 Now, the fact that $E$ necessarily admits a rational $2$-isogeny allows us to rule out many other $p$-torsion possibilities. For instance, if $13 \mid \# E(\QA)_{\mathrm{tors}}$, then $E$ must admit a rational $13$-isogeny, and so $E$ must admit a rational $26$-isogeny, but this is impossible by Theorem \ref{thm:n-isogeny-list}. Since $E[13]$ cannot be contained in $E(\QA)$, we see that $13 \nmid \# E(\QA)_{\mathrm{tors}}$.

 Similarly, if $E$ were to admit a rational $7$-isogeny (resp. $5$-isogeny), then $E$ would have to admit a rational $14$-isogeny (resp.~$10$-isogeny). By investigating \cite[Tables 3 and 4]{L-R13}, we see that this cannot occur for a curve with $j(E)=1728$. Since $E[5] \not\subseteq E(\QA)$ by Lemma \ref{lem:cyclotomic-genA4} and $E[7] \not\subseteq E(\QA)$ by Proposition \ref{prop:case-p-7}, we conclude that $E$ does not have any $5$- or $7$-torsion defined over $\QA$.

 It remains to consider $p=3$. By \cite[Table 3]{L-R13}, a curve with $j(E)=1728$ cannot have a $6$-isogeny, so if $3 \mid \# E(\QA)_{\mathrm{tors}}$ then by Lemma \ref{lem:torsion-implies-isogeny} we must have $E[3] \subseteq E(\QA)$. On the other hand, the $3$-division polynomial of $E_s$ is $3x^4+6sx^2-s^2$. This polynomial has discriminant $-2^{12}\cdot 3^3 \cdot s^6$, and it is irreducible. To see that it is irreducible one can argue as follows. First of all it has no roots since $E_s$ has no $3$-isogenies, so the only possibility is that it factors as a product of two quadratic terms. Write $s=3s'$ with $s' \in \QQ$ so that it suffices to prove the irreducibility of the monic polynomial $x^4 + 6s'x^2 - 3s'^2$ instead. We want to show that
 \begin{align*}
 x^4 + 6s'x^2 - 3s'^2 = (x^2+a_1x+a_2)(x^2+b_1x+b_2) = \cdots \\
 = x^4 + (a_1 + b_1)x^3 + (a_1b_1 + a_2 + b_2)x^2 + (a_2b_1 + a_1b_2)x + a_2b_2
 \end{align*}
 has no solutions with $a_1,a_2,b_1,b_2 \in \QQ$. This is indeed the case since the cubic term implies $a_1=-b_1$, after which the linear term implies either $a_1=0$ or $a_2=b_2$.
 \begin{itemize}
 	\item [1)] If $a_1=0$, then $x^4 + 6s'x^2 - 3s'^2 = (x^2+a_2)(x^2+b_2)$, in particular this means that $-a_2$ is a root of $y^2 + 6s'y - 3s'^2$ which is irreducible as soon as $s' \neq 0$ so this cannot happen.
 	\item [2)] If $a_2=b_2$, then looking at the constant term gives $a_2^2=-3s'^2$ which has no solutions as soon as $s' \neq 0$.
 \end{itemize}
 So the polynomial is indeed irreducible.
 By a standard result in Galois theory (c.f. \cite{Conrad}), the Galois group of this polynomial has exponent dividing $6$ if and only if its discriminant is a square in $\Q$, hence we may conclude by Corollary \ref{cor:genA4-exponent-divides-six} that $E(\QA)$ does not have any $3$-torsion.

 Our discussion in this section has proved the following results.

 \begin{prop}\label{prop:CMj1728}
 There are only two possible torsion structures over $\QA$ for an elliptic curve $E / \QQ$ with $j(E)=1728$. These structures are
 \[
 \ZZ / 2 \ZZ \oplus \ZZ / 2 \ZZ  \quad \text{and} \quad  \ZZ / 4 \ZZ \oplus \ZZ / 4\ZZ,
 \]
 and they are realized by the curves \href{http://www.lmfdb.org/EllipticCurve/Q/64a4}{\tt 64a4} and \href{http://www.lmfdb.org/EllipticCurve/Q/256c1}{\tt 256c1}, respectively.
 \end{prop}

 We collect the results of Sections \ref{sec:j0} and \ref{sec:j1728} in Table \ref{table:CM-special}.

 \begin{center}
   \begin{table}
 \begin{tabular}{|c|c||c|c|}\hline
 $E/ \QQ$  & $E(\QA)_{\mathrm{tors}}$  & $E/ \QQ$  & $E(\QA)_{\mathrm{tors}}$    \\\hline\hline
 \href{http://www.lmfdb.org/EllipticCurve/Q/54b3}{\tt 27a1}  & $\ZZ/3\ZZ \oplus \ZZ / 9\ZZ$ &  \href{http://www.lmfdb.org/EllipticCurve/Q/36a1}{\tt 36a1 } & $\ZZ / 2 \ZZ \oplus \ZZ / 6 \ZZ$ \\\hline
 \href{http://www.lmfdb.org/EllipticCurve/Q/64a4}{\tt 64a4}  &$\ZZ/4\ZZ \oplus \ZZ / 4\ZZ$  &\href{http://www.lmfdb.org/EllipticCurve/Q/108a1}{\tt 108a1}  & $\ZZ / 3\ZZ$ \\\hline
 \href{http://www.lmfdb.org/EllipticCurve/Q/256c1}{\tt 256c1}  & $\ZZ /2\ZZ \oplus \ZZ / 2\ZZ$ &  &  \\\hline
 \end{tabular}
 \caption{\label{table:CM-special} Torsion structures of CM curves $j = 0$ or $1728$}
 \end{table}
 \end{center}

 \subsection{The case when $E$ does not have complex multiplication}\label{sec:nonCM}

 Having classified torsion for CM curves in Section \ref{sec:CM}, it remains to consider curves without complex multiplication; in particular, the isomorphism type of $E(\QA)_\mathrm{tors}$ depends only on $j$. Before continuing, let us observe that the results of Sections \ref{sec:pPrimary} and \ref{sec:CM} combine to give us the following useful corollaries.

\begin{corollary}\label{cor:odd-p-parts}
If $E / \QQ$ is an elliptic curve and $p$ is an odd prime such that $E(\QA)(p)$ is nontrivial, then $p \in \{3,5,7,13 \}$ and $E$ has a rational $p$-isogeny.
\end{corollary}

\begin{corollary}\label{cor:2-parts}
Let $E / \QQ$ be an elliptic curve, let $\Delta(E)$ denote its discriminant, and suppose that $E(\QA)(2)$ is nontrivial. If $E$ does not admit a rational $2$-isogeny, then $\Delta(E)$ is a nonzero rational square, and in this case $E(\QA)_{\mathrm{tors}}$ contains a subgroup isomorphic to $\ZZ / 4 \ZZ \oplus \ZZ / 4 \ZZ$.
\end{corollary}

The proof of Corollary \ref{cor:2-parts} breaks down into two pieces depending on if the elliptic curve has CM or not. For non-CM elliptic curves curves Corollary \ref{cor:2-parts} follows immediately from the modular interpretation of the modular curves in Table \ref{table:2-primary} (available in the Rouse--Zeurick-Brown database) as well as the relationships between these modular curves illustrated in Figure \ref{fig:RZB}. For CM elliptic curves, Corollary \ref{cor:2-parts} follows immediately from the work done in Section \ref{sec:CM}.

\subsubsection{The case when $13$ divides $\# E(\QA)_{\mathrm{tors}}$}\label{full-torsion-13} This case is extremely easy to handle using the previous corollaries.

\begin{prop}\label{prop:p13final} Let $E / \QQ$ be an elliptic curve such that $13$ divides $\# E(\QA)_{\mathrm{tors}}$. Then \newline $E(\QA)_{\mathrm{tors}} \simeq \ZZ / 13\ZZ$.
\end{prop}

\begin{proof}
By Corollary \ref{cor:odd-p-parts}, if $13$ divides $\# E(\QA)_{\mathrm{tors}}$ then $E$ must admit a $13$-isogeny. The same corollary implies that if any other $p$-primary component is nontrivial for $p$ an odd prime, then $E$ must admit a rational $13p$-isogeny. Now Theorem \ref{thm:n-isogeny-list} implies that this is impossible.

The same argument shows that $E$ cannot admit a rational $2$-isogeny, so if $E(\QA)(2)$ is nontrivial, then by Corollary \ref{cor:2-parts}, $E(\QA)_{\mathrm{tors}}$ contains a subgroup isomorphic to $\ZZ / 4 \ZZ \oplus \ZZ / 4 \ZZ$ and $\Delta(E)$ is a nonzero rational square. From \cite[Table 3]{L-R13} we see that, since $E$ has a $13$-isogeny, it must be isomorphic to
\begin{align*}
E_t \colon y^2 + xy &= x^3 - \frac{36t}{(t^2 + 6 t + 13) (t^6 + 10 t^5 + 46 t^4 + 108 t^3 + 122 t^2 + 38 t - 1)^2} x\\
     &- \frac{t}{(t^2 + 6 t + 13) (t^6 + 10 t^5 + 46 t^4 + 108 t^3 + 122 t^2 + 38 t - 1)^2}
\end{align*}
over $\QA$ for some $t \in \QQ$. This curve has discriminant
\[
\Delta(E_t) = \frac{t (t^2 + 5t + 13)^2  (t^4 + 7t^3 + 20t^2 + 19t + 1)^6}{(t^2 + 6t + 13)^3  (t^6 + 10t^5 + 46t^4 + 108t^3 + 122t^2 + 38t - 1)^6}
\]
which is a nonzero rational square precisely when
\[
t(t^2+6t+13)
\]
is a nonzero rational square. Thus, we seek nonsingular, noncuspidal rational points on the elliptic curve
\[
C \colon s^2 = t(t^2+6t+13),
\]
which is the curve with Cremona label \href{http://www.lmfdb.org/EllipticCurve/Q/52a1}{\tt 52a1}. The only rational points on this curve are the point at infinity and the cusp $(0,0)$. Hence there are no elliptic curves with a $13$-isogeny and square discriminant, and so there are no elliptic curves with a $13$-isogeny and a nontrivial $2$-primary component.

 The result now follows from Proposition \ref{prop:case-p-13}.
\end{proof}

\subsubsection{The case when $7$ divides $\# E(\QA)_{\mathrm{tors}}$} Let us immediately observe that if $E(\QA)$ is nontrivial, then by the same argument as in Proposition \ref{prop:p13final}, the $13$- and $5$-primary components of $E(\QA)$ must both be trivial, but there do exist elliptic curves $E / \QQ$ with $14$- and $21$-isogenies, so we must consider the possibility that $E$ has a point of order $14$ or $21$ defined over $\QA$; note that by Theorem \ref{thm:n-isogeny-list}, it cannot possess both. As in Section \ref{full-torsion-13}, we must also consider the possibility that it has square discriminant. Let us first consider the possibility of $E$ possessing a $21$-isogeny.

\begin{lemma}\label{lem:z21}
If $E/\QQ$ is an elliptic curve such that both $E(\QA)(7)$ and $E(\QA)(3)$ are nontrivial, then $E(\QA)_{\mathrm{tors}} \simeq \ZZ /21 \ZZ$.
\end{lemma}

\begin{proof}From \cite[Table 3]{L-R13} we see that, up to $\bar{\QQ}$-isomorphism, there are only four curves which possess a $21$-isogeny, and they are represented by the curves with Cremona labels \href{http://www.lmfdb.org/EllipticCurve/Q/54b3}{\tt $162b1$}, \href{http://www.lmfdb.org/EllipticCurve/Q/162c1}{\tt 162c1}, \href{http://www.lmfdb.org/EllipticCurve/Q/162b3}{\tt $162b3$}, and \href{http://www.lmfdb.org/EllipticCurve/Q/162c3}{\tt $162c3$}. Checking each of these with Magma, we obtain the claimed result.
\end{proof}

It remains to consider the possibility that $E(\QA)(2)$ is nontrivial.  We have two cases depending on whether $E$ posseses a rational $2$-isogeny. Let us first assume that this is not the case so by Corollary \ref{cor:2-parts}, if $E(\QA)(2)$ is nontrivial then it must in fact be isomorphic to $\ZZ / 4 \ZZ \oplus \ZZ / 4 \ZZ$.

\begin{lemma}\label{lem:z4-z28}
There are infinitely many elliptic curves $E/ \QQ$ such that
\[
E(\QA)_{\mathrm{tors}} \simeq  \ZZ / 4\ZZ \oplus \ZZ / 28 \ZZ.
\]
\end{lemma}
\begin{proof}
Recall from Proposition \ref{prop:case-p-7} that if $7$ divides $\# E(\QA)_{\mathrm{tors}}$, then $E$ has
\[
j(E)= \frac{(t^2 + 13t + 49)(t^2 + 5t + 1)^3}{t}
\]
for some $t\in \QQ$. In this case $E$ is isomorphic over $\QA$ to
\[
E_t \colon y^2 + xy = x^3 - \frac{36t}{(t^4 + 14 t^3 + 63 t^2 + 70 t - 7)^2}x - \frac{t}{(t^4 + 14 t^3 + 63 t^2 + 70 t - 7)^2}
\]
and this curve has discriminant
\[
\Delta(E_t) = \frac{(t^4 + 14t^3 + 63t^2 + 70t - 7)^6}{t(t^2+5t+1)^6 (t^2+13t+49)^2}.
\]
Thus $\Delta(E)$ is a nonzero rational square if and only if $t$ is a nonzero rational square, and the result now follows from Corollary \ref{cor:2-parts}.
\end{proof}
\begin{remark}
Using the LMFDB database \cite{lmfdb}, we find that among the curves guaranteed to exist by Lemma \ref{lem:z4-z28}, the curve with smallest conductor is the one with Cremona label \href{http://www.lmfdb.org/EllipticCurve/Q/1922e2}{\tt 1922e2}.
\end{remark}

Finally, we must consider the case that $E$ possesses a $14$-isogeny.

\begin{lemma}\label{lem:z2-z14}
If $E/\QQ$ is an elliptic curve which possesses a $14$-isogeny, then $E(\QA)_{\mathrm{tors}} \simeq \ZZ /2 \ZZ \oplus \ZZ/14 \ZZ$.
\end{lemma}
\begin{proof}
From \cite[Table 3]{L-R13} we see that, up to $\bar{\QQ}$-isomorphism, there are only two curves which possess a $14$-isogeny, and they are represented by the curves with Cremona labels \href{http://www.lmfdb.org/EllipticCurve/Q/49a1}{\tt 49a1} and \href{http://www.lmfdb.org/EllipticCurve/Q/49a2}{\tt 49a2}. Checking each of these with Magma, we obtain the claimed result.
\end{proof}

We collect the results of this section in a single proposition.

\begin{prop}
Let $E / \QQ$ be an elliptic curve such that $E(\QA)(7)$ is nontrivial. Then $E(\QA)_{\mathrm{tors}}$ is isomorphic to one of the following groups:
\[
\ZZ / 7 \ZZ, \quad \ZZ /21 \ZZ, \quad \ZZ /2 \ZZ \oplus \ZZ/14 \ZZ, \quad \text{or} \quad  \ZZ / 4\ZZ \oplus \ZZ / 28 \ZZ.
\]
\end{prop}

\subsubsection{The case when $5$ divides $\# E(\QA)_{\mathrm{tors}}$} Let us now suppose that $5$ divides $\# E(\QA)_{\mathrm{tors}}$, so by Proposition \ref{prop:case-p-5} we know $E(\QA)(5)\simeq \ZZ / 5 \ZZ$. From the preceding sections, we see that the only other primes $p$ for which $E(\QA)(p)$ may be nontrivial are $p=2$ and $p=3$.

If $E(\QA)(3)$ is nontrivial, then by Corollary \ref{cor:odd-p-parts} we see that $E$ must possess a rational $15$-isogeny. The following argument will by now be familiar to the reader.

\begin{lemma}\label{lem:15isog}
If $E/\QQ$ is an elliptic curve which possesses a $15$-isogeny, then $E(\QA)_{\mathrm{tors}}$ is isomorphic to either $\ZZ / 3 \ZZ$ or $\ZZ / 15 \ZZ$.
\end{lemma}
\begin{proof}
From \cite[Table 3]{L-R13} we see that, up to $\bar{\QQ}$-isomorphism, there are only four curves which possess a $15$-isogeny, and they are represented by the curves with Cremona labels \href{http://www.lmfdb.org/EllipticCurve/Q/50a1}{\tt 50a1}, \href{http://www.lmfdb.org/EllipticCurve/Q/50a2}{\tt 50a2}, \href{http://www.lmfdb.org/EllipticCurve/Q/50a3}{\tt 50a3} and \href{http://www.lmfdb.org/EllipticCurve/Q/50a4}{\tt 50a4}. A computation in Magma confirms that over $\QA$, the curves \href{http://www.lmfdb.org/EllipticCurve/Q/50a1}{\tt 50a1} and \href{http://www.lmfdb.org/EllipticCurve/Q/50a2}{\tt 50a2} have torsion subgroup $\ZZ / 3 \ZZ$, while \href{http://www.lmfdb.org/EllipticCurve/Q/50a3}{\tt 50a3} and \href{http://www.lmfdb.org/EllipticCurve/Q/50a4}{\tt 50a4} have torsion subgroup $\ZZ / 15\ZZ$.
\end{proof}

Now we consider the possible $2$-torsion. By Corollary \ref{cor:2-parts}, if $E(\QA)(2)$ is nontrivial, then either $E$ admits a rational $2$-isogeny or $E(\QA)$ contains a subgroup isomorphic to $\ZZ / 4 \ZZ \oplus \ZZ / 4 \ZZ$. We begin by ruling out this second possibility.

\begin{lemma}\label{lem:z5-full4}
Suppose $E / \QQ$ is an elliptic curve such that $E(\QA)(5) \simeq \ZZ / 5 \ZZ$. Then $E[4] \not\subseteq E(\QA)$.
\end{lemma}

\begin{proof}
Let $E / \QQ$ be an elliptic curve such that $5 \mid \#E(\QA)_\tor$, so that by \cite[Appendix E]{L-Rbook} (see also Table \ref{tab:strong_cyclic_groups}) $E$ is isomorphic over $\QQ$ to a quadratic twist of
\[
E_t \colon y^2 + (1-t)xy -ty = x^3 - tx^2,
\]
which has discriminant
\[
\Delta(E_t) = t^5(t^2 - 11t - 1).
\]
By Corollary \ref{cor:2-parts}, if $E[4] \subseteq E(\QA)$, then $\Delta(E)$ must be a square. Up to squares $\Delta(E)$, $\Delta(E_t)$ and $t(t^2-11t-1)$ are the same, and hence the latter must be a square as well. As in the proof of Proposition \ref{prop:p13final}, we find that the curve defined by $s^2=t(t^2-11t-1)$ is \href{http://www.lmfdb.org/EllipticCurve/Q/50a1}{\tt 20a4}, which has only the point at infinity and $(0,0)$, and the result follows.
\end{proof}

\begin{lemma}\label{lem:z5-z2-z4}
Suppose $E / \QQ$ is an elliptic curve such that $E(\QA)(5) \simeq \ZZ / 5 \ZZ$. Then $E(\QA)(2)$ is either trivial or isomorphic to $\ZZ / 2 \ZZ \oplus \ZZ / 2 \ZZ$.
\end{lemma}

\begin{proof}
In light of Lemma \ref{lem:z5-full4} and Proposition \ref{prop:2-tors}, it suffices to show that $E(\QA)(2)$ cannot contain a subgroup isomorphic to $\ZZ / 2 \ZZ \oplus \ZZ / 4 \ZZ$. By Corollary \ref{cor:RZBtable}, if $E(\QA)(2)$ contains a subgroup isomorphic to $\ZZ / 2 \ZZ \oplus \ZZ / 4 \ZZ$, then $E$ corresponds to a rational point on the modular curve labelled \href{http://users.wfu.edu//rouseja/2adic/X13.html}{$X_{13}$} in the notation of \cite{RZB}, which is more familiarly known as $X_0(4)$, so $E$ admits a rational $4$-isogeny. Then by Corollary \ref{cor:odd-p-parts}, $E$ also admits a rational $5$-isogeny, hence it must be isogenous to a curve which admits a $20$-isogeny, but this is impossible by Theorem \ref{thm:n-isogeny-list}.
\end{proof}

\begin{remark}
The elliptic curve with Cremona label \href{ http://www.lmfdb.org/EllipticCurve/Q/66c1}{\tt 66c1} satisfies $E(\QA)_\tor \simeq \ZZ / 2\ZZ \oplus \ZZ / 10 \ZZ$ and is the curve of smallest conductor with this property.
\end{remark}

\subsection{The case when $3$ divides $\# E(\QA)_{\mathrm{tors}}$} Let us now suppose that $3$ divides $\# E(\QA)_{\mathrm{tors}}$, so by Lemma \ref{lem:3x9} we know that $E(\QA)(3)$ is isomorphic to a subgroup of $\ZZ / 3 \ZZ \oplus \ZZ / 9 \ZZ$. From the preceding sections, we see that the only other prime $p$ for which $E(\QA)(p)$ may be nontrivial is $p=2$.

Let us first clear the playing field a bit with the following lemmas.

\begin{lemma}\label{lem:3-no4x4}
If $E / \QQ$ is an elliptic curve such that $E(\QA)(3)$ is nontrivial, then $E(\QA)(2)$ cannot contain a subgroup isomorphic to $\ZZ / 2 \ZZ \oplus \ZZ / 8 \ZZ$.
\end{lemma}

\begin{proof}
By Corollary \ref{cor:RZBtable}, if $E(\QA)(2)$ contains a subgroup isomorphic to $\ZZ / 2 \ZZ \oplus \ZZ / 8 \ZZ$, then $E$ corresponds to a rational point on the modular curve labelled \href{http://users.wfu.edu//rouseja/2adic/X36.html}{$X_{36}$} in the notation of \cite{RZB}, which is more familiarly known as $X_0(8)$, so $E$ admits a rational $8$-isogeny. Then by Corollary \ref{cor:odd-p-parts} $E$ also admits a rational $3$-isogeny, hence it would be isogenous to a curve which admits a $24$-isogeny, but this is impossible by Theorem \ref{thm:n-isogeny-list}.
\end{proof}

Thus, from Corollary \ref{cor:RZBtable}, the only possible nontrivial $2$-torsion structures in this case are
\[
\ZZ / 2 \ZZ \oplus \ZZ / 2 \ZZ, \quad \ZZ / 2 \ZZ \oplus \ZZ / 4 \ZZ, \quad \text{and} \ \ZZ / 4 \ZZ \oplus \ZZ / 4 \ZZ.
\]
In fact, we have the following refinement.

\begin{lemma}\label{lem:3-no4x4}
If $E / \QQ$ is an elliptic curve such that $E[3] \subseteq E(\QA)$, then $E(\QA)(2)$ cannot contain a subgroup isomorphic to $\ZZ / 2\ZZ \oplus \ZZ / 4 \ZZ$.
\end{lemma}

\begin{proof}
As in the proof of Lemma \ref{lem:z5-z2-z4}, if $E(\QA)(2)$ were to contain a subgroup isomorphic to $\ZZ / 2\ZZ \oplus \ZZ / 4 \ZZ$, then $E$ would admit a rational $4$-isogeny. But if $E[3] \subseteq E(\QA)$, then Lemma \ref{lem:full-3-torsion} implies that $E$ must also admit two independent $3$-isogenies, hence $E$ is isogenous to a curve which admits a $9$-isogeny, and so $E$ would in fact be isogenous to a curve which admits a $36$-isogeny, which is impossible by Theorem \ref{thm:n-isogeny-list}.
\end{proof}

\begin{remark}\label{rmk:z4-z12}
We still need to consider the cases where $E(\QA)(3) \simeq \ZZ / 3 \ZZ$. For instance, the curve with Cremona label \href{http://www.lmfdb.org/EllipticCurve/Q/30a2}{\tt 30a2} has $E(\QA) \simeq \ZZ / 4 \ZZ \oplus \ZZ / 12 \ZZ$.
\end{remark}

\begin{lemma}\label{lem:2-no3x9}
If $E / \QQ$ is an elliptic curve such that $E(\QA)(2)$ is nontrivial, then $E(\QA)(3)$ cannot be isomorphic to $\ZZ / 3 \ZZ \oplus \ZZ / 9 \ZZ$.
\end{lemma}

\begin{proof}
As in the proof of Lemma \ref{lem:z3-z9-example}, if $E(\QA)(3) \simeq \ZZ / 3 \ZZ \oplus \ZZ / 9 \ZZ$ then $E$ is isogenous to a curve which admits a rational $27$-isogeny, and if $E(\QA)(2)$ is nontrivial then $E$ also admits a rational $2$-isogeny, hence it must be isogenous to a curve which admits at rational $54$-isogeny, which is impossible by Theorem \ref{thm:n-isogeny-list}.
\end{proof}




\begin{lemma}\label{lem:z2xz6}
If $E(\QA)(3) \simeq \ZZ / 3 \ZZ$, then $E$ is isogenous to an elliptic curve $E_t$ with $j$-invariant
\[
j(E_t) = \frac{(t+27)(t+3)^3}{t}
\]
for some $t \in \QQ$. Furthermore, if $E(\QA)(2)$ is nontrivial, then we have
\[
E(\QA)_\tor \simeq \begin{cases}
 \ZZ / 2 \ZZ \oplus \ZZ / 6 \ZZ \ \text{or} \ \ZZ / 2 \ZZ \oplus \ZZ / 12 \ZZ  & \text{if $t$ is not a square}\\
\ZZ / 4 \ZZ \oplus \ZZ / 12 \ZZ & \text{if $t$ is a nonzero rational square}.
\end{cases}
\]
\end{lemma}

\begin{proof}
The $j$-invariant is that of $X_0(3)$ from \cite[Table 2]{L-R13}, since $E(\QA)(3) \simeq \ZZ / 3 \ZZ$ implies $E$ admits a rational $3$-isogeny. The discriminant of $E_t$ is given by
\[
\Delta(E_t) = \frac{t(t+3)^6(t+27)^2}{(t^2+18t-27)^6},
\]
hence $\Delta(E)$ is a nonzero rational square if and only if $t$ is a nonzero rational square. The second statement now follows from Corollary \ref{cor:2-parts}.
\end{proof}

\begin{remark}
Indeed, the minimal twists of the curves $E_1$ and $E_2$ have Cremona labels $196a1$ and $1682f1$, respectively, and Magma confirms that
\[
E_1(\QA)_\tor \simeq \ZZ / 4 \ZZ \oplus \ZZ / 12 \ZZ \quad \text{and} \quad
E_2(\QA)_\tor \simeq \ZZ / 2 \ZZ \oplus \ZZ / 6 \ZZ.
\]
\end{remark}

\begin{remark}
We point out that the above argument assumes $E(\QA)(3)$ is {\it precisely} $ \ZZ / 3 \ZZ$. Since all it uses is the existence of a $3$-isogeny, it is unable to distinguish such a curve from one with a larger $3$-primary component.
\end{remark}

\begin{lemma}\label{lem:z2z18}
Suppose $E / \QQ$ is an elliptic curve, then $E(\QA)_\tor \simeq \ZZ / 2 \ZZ \oplus \ZZ / 18 \ZZ$ if and only if its  $j$-invariant is of the form
\[
j(E_t) = \frac{(t^3-2)^3(t^9-6t^6-12t^3-8)^3}{t^9 (t^3-8)(t^3+1)^2} \in \QQ^\times.
\]

\end{lemma}

\begin{proof}
Our curve $E / \QQ$ has the desired $\QA$-torsion structure if and only if $E$ admits a rational $18$-isogeny. We note that when $E(\QA)(3) \simeq \ZZ / 9\ZZ$, Figure \ref{fig:RZB} demonstrates that $E(\QA)(2)$ cannot be any larger than $\ZZ / 2 \ZZ \oplus \ZZ / 2 \ZZ$, since then (via the modular curve $X_{13}$) it would possess a rational $4$-isogeny, but no elliptic curve over $\QQ$ admits a rational $36$-isogeny. The $j$-invariant in the statement is that of $X_0(18)$ from \cite[Table 2]{L-R13}.
\end{proof}

\section{Parameterizations for each possible torsion structure.}\label{sec:TParam}

We conclude the study of torsion of generalized $A_4$ extensions by parameterizing each of the possible torsion structures that occur for an elliptic curve $E/\QQ$ base extended to $\QQ(A_4^\infty)$. Since the torsion subgroup of $E(\QA)$ depends only on the $j$-invariant of $E$ (unless $j(E) =0$ or 1728, and these two cases have been dealt with in Section \ref{sec:CM}) it is enough to give a complete description of the sets
$$S_T = \{j(E) : E(\QA)_\tor \simeq T\}$$
for each of the 26 possible torsion structure given in Theorem \ref{thm:main}.

To describe each set $S_T$ we give sets $F_T$ of rational functions $j(t)$ such that $E(\QA)$ contains a subgroup isomorphic to $T$ if and only if there is a function $j(t)\in F_T$ and a rational number $r\in\QQ$ such that $j(E) = j(r)$. Following the notation laid out in \cite{Q(3)}, we will let $\mathcal{T}$ be the set of all the possible 26 torsion structures up to isomorphism for $E(\QA)$, and we put a partial order on $\mathcal{T}$ given by $T_1\leq T_2$ if $T_2$ has a subgroup isomorphic to $T_1$. Next, for any elliptic curve $E/\QQ$ with $j\neq 0$ or 1728 we define $\mathcal{T}(E)$ to be the set of $T\in\mathcal{T}$ such that $j(E)$ is in the image of one of the rational functions in $F_T$.

\begin{thm}\label{thm:main_param}
Let $E/\QQ$ be an elliptic curve with $j(E) \neq 0$ or $1728$. Then the set $\mathcal{T}(E)$ has a unique maximal element $T(E)$ with respect to the partial order on $\mathcal{T}$ and $E(\QA)_\tor\simeq T(E)$.
\end{thm}

\begin{remark}
The set $\mathcal{T}(E)$ need not contain every $T\in\mathcal{T}$ such that $T\leq T(E)$. There are infinitely many examples of this since there are only finitely many elliptic curves $E/\QQ$ such that $T(E) \simeq \ZZ / 2 \ZZ \oplus \ZZ / 14 \ZZ$, while there are infinitely many elliptic curves $E'/\QQ$ such that $T(E')\simeq \ZZ / 4 \ZZ \oplus \ZZ / 28 \ZZ$. This occurs because in order for $E(\QA)(2) \simeq \ZZ/2\ZZ\oplus\ZZ/2\ZZ$, $E$ must have a 2-isogeny, while it is possible for $E'(\QA)(2) \simeq \ZZ/4\ZZ\oplus\ZZ/4\ZZ$ without having any isogenies. For more details about this situation see Corollary \ref{cor:RZBtable} and \cite{RZB}.
\end{remark}

The proof of Theorem \ref{thm:main_param} follows by the exact same argument (\textit{mutatis mutandis}) as in \cite[Theorem 7.1]{Q(3)} and so we omit it here for the sake of brevity. Instead we give a table of all of the sets $F_T$.

Justification for the maps in Table \ref{tab:TParams} are either given in Sections \ref{sec:pPrimary} and \ref{sec:Torsion} or they are constructed as the fiber product of the relevant $j$-maps. The construction of these curves can be found in \cite{A4Code}.

{
\begin{table}
\tabulinesep=1.1mm
\centering
\begin{tabu}{l|l}

$T$               & $j(t)$ \\ \hline
$\{\mathcal{O}\}$ & $t$    \\
$\ZZ / 3 \ZZ$            & $\frac{(t+27)(t+3)^3}{t}$ \\
$\ZZ / 5 \ZZ$ & $\frac{(t^4 - 12 t^3 + 14t^2 + 12t + 1)^3}{t^5 (t^2-11t -1)}$\\
$\ZZ / 7 \ZZ$ & $\frac{(t^2+13t+49)(t^2+5t+1)^3}{t}$ \\
$\ZZ / 9 \ZZ$ & $\frac{t^3(t^3-24)^3}{t^3-27}$ \\
$\ZZ / 13\ZZ$ & $\frac{(t^4 - t^3 + 5t^2 + t + 1)(t^8 - 5t^7 + 7t^6 - 5t^5 + 5t^3 + 7t^2 + 5t + 1)^3}{t^{13}(t^2-3t-1)}$ \\
$\ZZ / 15 \ZZ$ & $\{ -\frac{121945}{32} , \frac{46969655}{32768} \}$ \\
$\ZZ / 21 \ZZ$ & $\{ -\frac{140625}{8}, \frac{3375}{2}, -\frac{1159088625}{2097152}, -\frac{189613868625}{128}  \}$ \\
$\ZZ / 2 \ZZ \oplus \ZZ / 2 \ZZ$ & $\frac{t^3}{t + 16}$ \\
$\ZZ / 2 \ZZ \oplus \ZZ / 4 \ZZ$ & $\frac{(t^2-48)^3}{(t-8)(t+8)}$ \\
$\ZZ / 2 \ZZ \oplus \ZZ / 6 \ZZ$ & $\frac{(t+6)^3(t^3+18t^2+84t+24)^3}{t(t+8)^3(t+9)^2}$ \\
$\ZZ / 2 \ZZ \oplus \ZZ / 8 \ZZ$ & $\frac{(t^4-16t^2+16)^3}{t^2(t-4)(t+4)}$ \\
$\ZZ / 2 \ZZ \oplus \ZZ / 10 \ZZ$ &  $\frac{(u^6 + 4u^5 - 16u+16)^3}{u^5(u-1)^2(u+4)}$, \quad \text{where}\ $u=t-\frac{1}{t}$ \\
$\ZZ / 2 \ZZ \oplus \ZZ / 12 \ZZ$ & $\frac{(t^2-3)^3(t^6-9t^4+3t^2-3)^3}{(t-3)(t-1)^3t^4(t+1)^3(t+3)}$ \\
$\ZZ / 2 \ZZ \oplus \ZZ / 14 \ZZ$ & $\{ -3375, 16581375  \}$ \\
$\ZZ / 2 \ZZ \oplus \ZZ / 16 \ZZ$ & $ \frac{(t^{16} - 8t^{14} + 12t^{12} + 8t^{10} - 10t^8 + 8t^6 + 12t^4 - 8t^2 + 1)^3}{t^{16} (t-1)^4(t+1)^4(t^2+1)^2(t^2-2t-1)(t^2+2t-1)}$  \\
$\ZZ / 2 \ZZ \oplus \ZZ / 18 \ZZ$ & $\frac{(t^3-2)^3(t^9-6t^6-12t^3-8)^3}{t^9 (t^3-8)(t^3+1)^2}$ \\
$\ZZ / 3 \ZZ \oplus \ZZ / 3 \ZZ $ & $\frac{27(t+1)^3(t+3)^3(t^2+3)^3}{t^3 (t^2+3t+3)^3} $ \\
$\ZZ / 3 \ZZ \oplus \ZZ / 9 \ZZ$ & $\{ 0 \}$ \\
$\ZZ / 4 \ZZ \oplus \ZZ / 4 \ZZ$ & $t^2 + 1728, -\frac{4(4t^2-8t+1)^3(4t^2+8t+1)^3}{t^2(4t^2+1)^4}$ \\
$\ZZ / 4 \ZZ \oplus \ZZ / 8 \ZZ$ & $\frac{256(t^4-t^2+1)^3}{(t-1)^2t^4(t+1)^2}, \frac{-4 (t^4 - 8 t^3 + 2 t^2 + 8 t + 1)^3 (t^4 + 8 t^3 + 2 t^2 - 8 t + 1)^3}{t^2 (t - 1)^2 (t + 1)^2 (t^2 + 1)^8}$ \\
$\ZZ / 4 \ZZ \oplus \ZZ / 12 \ZZ$ & $\frac{729t^8 + 756t^6 + 270t^4 + 36t^2 + 1}{t^6}$ \\
$\ZZ / 4 \ZZ \oplus \ZZ / 16 \ZZ$ & $ \frac{2^8 (u^2-u+1)^3}{u^2(u-1)^2} $ , \quad \text{where}\ $u=\left(\frac{t-\frac{1}{t}}{2}\right)^4$ \\
$\ZZ / 4 \ZZ \oplus \ZZ / 28 \ZZ$ & $\frac{(t^4+13t^2+49)(t^4+5t^2+1)^3}{t^2}$ \\
$\ZZ / 6 \ZZ \oplus \ZZ / 6 \ZZ$ & $\frac{(t^3 - 57t^2 + 3t - 1)^3(53t^3 + 3t^2 - 3t + 1)^3(8587 t^6 - 8214t^5 + 2283t^4 + 304t^3 - 39t^2 - 6t + 1)^3}{729(t-1)^3(4t-1)^6t^6(5t+1)^3(43t^2 - 8t + 1)^3(7t^2 + t + 1)^6}$ \\
$\ZZ / 8 \ZZ \oplus \ZZ / 8 \ZZ$ & $\frac{16 (t^4 - 2 t^3 + 2 t^2 + 2 t + 1)^3 (t^4 + 2 t^3 + 2 t^2 - 2 t + 1)^3}{(t - 1)^4 t^4 (t + 1)^4 (t^2 + 1)^4}$
\end{tabu}
\caption{Parameterizations $j(t)$ of the $\bar{\QQ}$-isomorphism classes of elliptic curves $E / \QQ$ corresponding to the isomorphism type of $E(\QA)$.}\label{tab:TParams}
\end{table}
}

\section{Torsion over the compositum of all $A_4$-extensions of $\QQ$}\label{sec:compositum}

The last task is to determine what subgroups (up to isomorphism) occur as the torsion subgroup of an elliptic curve over $\QQ$ base extended to the compositum of all $A_4$ extensions of $\QQ$. For the rest of this section we will let $\QQ_{A_4}$ be the compositum of all $A_4$ extensions of $\QQ$.

Before attempting this classification we recall from the proof of Proposition \ref{prop:F_ne_A4} that $\QQ_{A_4}$ does not contain any quadratic extensions of $\QQ$, and because of this it is possible to have two elliptic curves $E$ and $E'$ with the same $j$-invariant but different torsion subgroups when base-extended to $\QQ_{A_4}$.

\begin{example}
Let $E$ be the elliptic curve with Cremona reference \href{http://www.lmfdb.org/EllipticCurve/Q/44a1}{\texttt{44a1}} and let $E'$ be the elliptic curve with Cremona reference \href{http://www.lmfdb.org/EllipticCurve/Q/176c1}{\texttt{176c1}}. In this case $j(E) = j(E')$ and $$E(\QA)_\tor \simeq E'(\QA)_\tor\simeq \ZZ/3\ZZ.$$
Using \cite{lmfdb}, we see that $E(\QQ)[3] \simeq \ZZ/3\ZZ$ while $E'(\QQ)[3]$ is trivial. The curves $E$ and $E'$ are quadratic twists of each other and both have a 3-isogeny. In the case of $E$, the kernel of its 3-isogeny is defined over $\QQ$ while in the case of $E'$ the kernel of its 3-isogeny is defined over $\QQ(i)$. Therefore, $E$ and $E'$ are $\overline{\QQ}$-isomorphic, but $E(\QQ_{A_4})_\tor\not\simeq E'(\QQ_{A_4})_\tor.$
\end{example}

While this also happens when considering base-extension to $\QQ(A_4^\infty)$ when $j(E)=0$ or 1728, it can happen in many more instances when considering base extension to $\QQ_{A_4}$, and because of this, replicating Table \ref{tab:TParams} in this context is not possible.

We also note that, while it is true that $E(\QQ_{A_4})_\tor \subseteq E(\QQ(A_4^\infty))_\tor$ for every elliptic curve $E/\QQ$ (since $\QQ_{A_4}\subseteq \QQ(A_4^\infty)$), it is not the case that every group which arises as the torsion subgroup of an elliptic curve base-extended to $\QQ_{A_4}$ also arises as the torsion subgroup of an elliptic curve base-extended to $\QQ(A_4^\infty)$, as shown in the following example.

\begin{example}
Let $E$ be the elliptic curve with Cremona reference \href{http://www.lmfdb.org/EllipticCurve/Q/46a1}{\texttt{46a1}}. From Table 1, we know that $E(\QA)_\tor \simeq \ZZ/2\ZZ\oplus\ZZ/2\ZZ$ and from \cite{lmfdb} we know that $\QQ(E[2]) = \QQ(\sqrt{-23})$. Since $\QQ(E[2])\cap\QQ_{A_4} = \QQ$ we know that $E(\QQ_{A_4})_\tor = E(\QQ)_\tor \simeq \ZZ/2\ZZ,$ but according to Theorem \ref{thm:main} this group does not occur over $\QA$.
\end{example}

\begin{lemma}\label{lem:is_strong_A4}
Let $F/\QQ$ be a number field and $\widetilde{F}$ the Galois closure of $F$. Then $F\subseteq \QQ_{A_4}$ if and only if $\Gal(\widetilde{F}/\QQ)$ is of strong $A_4$-type.
\end{lemma}

\begin{proof}
By the Galois correspondence this just follows from its group theoretic counterpart \cref{cor:strong-Dpq}.
\end{proof}

\begin{cor}\label{cor:abel==>exp3}
If $F/\QQ$ is an abelian extension of $\QQ$ such that $F\subseteq \QQ_{A_4}$, then $\Gal(F/\QQ)\simeq (\ZZ/3\ZZ)^k$ for some $k$.
\end{cor}

\begin{proof}
This is an immediate consequence of Lemmas \ref{lem:strong_Dpq} and  \ref{lem:is_strong_A4}.
\end{proof}

\begin{cor}\label{cor:strong_root_of_1}
The only roots of unity in $\QQ_{A_4}$ are $\pm1$.
\end{cor}

\subsection{Possible 2-power torsion growth when base-extended to $\QQ_{A_4}$}

Let $E/\QQ$ be an elliptic curve. Then $\Gal(\QQ(E[2])/\QQ)\simeq S_3, \ZZ/3\ZZ, \ZZ/2\ZZ$ or it is trivial. Using Corollary \ref{cor:no_quad_subfields}, we can see that the only way that $\QQ(E[2])\subseteq \Q_{A_4}$ is if $\Gal(\QQ(E[2])/\QQ)\simeq \ZZ/3\ZZ,$ or it is trivial. Thus the only way that the $2$-torsion can grow when base-extended to $\Q_{A_4}$ is to have $E(\QQ)[2]$ trivial and $\Gal(\QQ(E[2])/\QQ)\simeq \ZZ/3\ZZ$. We notice here that the second condition corresponds to $E$ having square discriminant. Since $\QQ(i)\not\subseteq \Q_{A_4}$, we know that $\QQ(E[4])\not\subseteq \QQ_{A_4}$ and we cannot gain full 4-torsion when base extended to $\QQ_{A_4}$.


Next, we notice that we in fact cannot gain a point of order 4 when base-extending to $\QQ_{A_4}$. To see this, suppose that $E$ does gain a point of order $4$ over $\QQ_{A_4}$. If $E(\QQ)[2]$ were not trivial, then this would have to happen over a quadratic extension of $\QQ$ since $E$ would have to have a $4$-isogeny, and $\Q_{A_4}$ contains no such subextensions. On the other hand, if $E(\QQ)[2]$ were trivial, then $E(\QQ_{A_4})(2)\simeq \ZZ/2\ZZ\oplus \ZZ/4\ZZ$ and so $E$ would possess a $2$-isogeny, but this would imply that $E(\QQ)[2]$ was nontrivial, giving a contradiction.

Before continuing we state proposition that will prove to be useful.

\begin{prop}{\rm \cite[Proposition 4.8]{GonNajman}}\label{prop:2_torsion_growth}
Let $E/F$ be an elliptic curve defined over a number field $F$ with fixed algebraic closure $\overline{F}$, $n$ a positive integer, $P\in E(\overline{F})$ be a point of order $2^{n+1}$ and let $\widehat{F(P)}$ be the Galois closure of $F(P)$ over $F(2P)$. Then $[F(P) : F(2P)]$ divides 4 and $\Gal(\widehat{F(P)}/F(2P))$ is either isomorphic to $\Z/2\Z$, $\Z/2\Z \times \Z/2\Z$, $D_4$, or it is trivial.
\end{prop}

Proposition \ref{prop:2_torsion_growth} combined with Corollary \ref{cor:no_quad_subfields} and the previous discussion immediately yields to the following two lemmas.

\begin{lemma}
Let $E/\QQ$ be an elliptic curve. Then either $E(\QQ_{A_4})(2) = E(\QQ)(2)$ or $E(\QQ)[2]$ is trivial and $E(\QQ_{A_4})(2) = E[2]$. The $2$-torsion grows exactly when $E(\QQ)[2]$ is trivial and $E$ has square discriminant.
\end{lemma}

\begin{lemma}\label{lem:strong_full_2}
If $E/\QQ$ is an elliptic curve, then $\QQ(E[2])\subseteq \QQ_{A_4}$ if and only if ${\Delta(E)}$ is a rational square.
\end{lemma}

\subsection{The possible prime-to-2 torsion}

From Corollary \ref{cor:strong_root_of_1} and the Weil pairing, if we let $E/\QQ$ be an elliptic curve and $G$ be the maximal subgroup of $E(\QQ_{A_4})$ of odd order, we get that $G$ must be cyclic and isomorphic to $\ZZ/N\ZZ$ for some odd $N$. By Theorem \ref{thm:main}, the only possibilities are $N = 1,3,5,7,9,13,15,$ or $21$.

We start by noticing that the only way $E$ can have a point of order 3 or 5 defined over $\QQ_{A_4}$ is if it already had a point of order 3 or 5 defined over $\QQ$. Therefore, since there are no elliptic curves $E/\QQ$ with a point of order 15 defined over $\QQ$, N cannot be 15. Further, if $N= 21$ then $E$ must have a rational point of order $3$ and a $7$-isogeny whose kernel has a generator defined over a cubic extension of $\QQ$. From \cite[Theorem 1]{najman-cubic}, we know that there is exactly one such elliptic curve up to $\QQ$-isomorphism and that is the curve with Cremona reference \href{http://www.lmfdb.org/EllipticCurve/Q/162b1}{\texttt{162b1}}. Therefore, all that is left to do is classify when the other possible $N>1$ occur, and the results are summarized in Table \ref{tab:strong_cyclic_groups}. The models in this table for $N=3,5$ can be found in \cite[Appendix E]{L-Rbook}, while the models for $N=7,13$ can be obtained from \cite{Zywina}, and the model for $N=9$ is the model for a curve with a $9$-isogeny twisted so that the associated $3$-isogeny has a $\QQ$-rational kernel.

\begin{table}
\tabulinesep=1.3mm
\centering
\begin{tabu}{l|l}
N & model for the generic curve with a cyclic group of order $N$ in $E(\QQ_{A_4})_\tor$\\\hline
3  &  $y^2+axy+by = x^3$  (can take $a=1$ if $j(E)\neq 0$) \\\hline
5  &  $y^2+(1-t)xy-ty=x^3-tx^2$   \\\hline
7  & \footnotesize  $y^2= x^3-27(t^2 + 5t + 1)(t^2 + 13t + 49)^3x $\\
& \footnotesize \phantom{$y^2 = x^3 -$}$+54(t^2 + 13t + 49)^4(t^4 + 14t^3 + 63t^2 + 70t - 7)$  \\\hline
9  &\footnotesize $y^2 + txy + y = x^3$    \\\hline
13  &  \footnotesize $y^2 = x^3 - 27(t^4 -t^3 +5t^2 +t+1)^3 (t^8 -5t^7 +7t^6 -5t^5 +5t^3 +7t^2 +5t+1) x$ \\
 & \footnotesize \phantom{$y^2 = x^3$} $+54(t^2 +1)(t^4 -t^3 +5t^2 +t+1)^4 (t^{12} -8t^{11} +25t^{10} -44t^9 +40t^8 +18t^7 $\\
 & \footnotesize \phantom{$y^2 = x^3$} $-40t^6 -18t^5 +40t^4 +44t^3 +25t^2 +8t+1)$   \\
\end{tabu}
\caption{}\label{tab:strong_cyclic_groups}
\end{table}

\subsection{Torsion subgroups that occur over $\QQ_{A_4}$} All that is left now is to determine what combinations of $2$-powered torsion can occur with the prime-to-$2$ options. Before doing this we give the following example.

\begin{example}\label{ex-Strong_A4_14}
Let $E$ be the elliptic curve with Cremona reference \href{http://www.lmfdb.org/EllipticCurve/Q/49a4}{\texttt{49a4}} and let $F$ be the splitting field of $f(x) = x^3-x^2-2x+1$. Then $F\subseteq\QQ_{A_4}$ since $\Gal(F/\QQ) \simeq \ZZ/3\ZZ$ and $E(F)_\tor \simeq \ZZ/14\ZZ$, and therefore $E(\QQ_{A_4})_\tor \simeq \ZZ/14\ZZ$ as well. Similarly if $E'$ is the elliptic curve with Cremona reference \href{http://www.lmfdb.org/EllipticCurve/Q/49a3}{\texttt{49a3}}, then $E(\QQ_{A_4})_\tor \simeq \ZZ/14\ZZ$ as well. These are the only elliptic curves with a point of order 14 over $\QQ_{A_4}$ without full 2-torsion defined over $\QQ_{A_4}$.
\end{example}

\begin{thm}\label{thm:main_strong}
Let $E/\QQ$ be an elliptic curve. The torsion subgroup $E(\QQ_{A_4})_\tor$ is finite and
\[
E(\QQ_{A_4})_\tor\simeq
\begin{cases}
\ZZ/M\ZZ & \hbox{with } 1\leq M \leq 10 \hbox{ or } M = 12, 13, 14, 18, 21 \hbox{ or}\\
\ZZ/2\ZZ \oplus \ZZ/2M\ZZ & \hbox{with }1\leq M \leq 4 \hbox{ or } M=7.\\
\end{cases}
\]
\end{thm}

\begin{proof}
Notice that the above list is the same as the list of torsion structures that occur already over $\QQ$ (c.f. Theorem \ref{thm:mazur}) with the exception of $\Z/M\Z$ for $M=13, 14, 18,$ and $21$ and $\Z/2\Z\oplus \Z/14\Z$. The fact that $7$ and $13$ torsion occurs follows immediately from \cref{tab:strong_cyclic_groups} where we give formulas for the generic elliptic curves with a 7- or 13-isogeny whose kernel is defined over a degree 3 extension of $\Q$. These formulas are taken from \cite{Zywina}. Next, the fact that $18$ occurs follows from the same table with the additional observation that the curve $y^2 = f(x,t)$ has rational 2 torsion if and only if $f(x,t)$ has a zero, and the curve $f(x,t)=0$ in this case is birational to $\PP^1$ (See Example \ref{ex-Strong_A4_18}). The case when $T = \Z/14\Z$ follows from Example \ref{ex-Strong_A4_14}, while the case when $T = \Z/2\Z\oplus \Z/14\Z$ follows from observing that the fiber product of the modular curves whose rational points parameterize elliptic curves with square discriminant ($X_2$ in the notation of \cite{RZB}) and the modular curve whose rational points correspond to parameterizing elliptic curves with a point of order 7 defined over a cubic field is isomorphic to $\PP^1$. The $\QQ$-points on this curve correspond to the elliptic curves over $\QQ$ whose torsion subgroup contains a subgroup isomorphic to $\Z/2\Z\oplus\Z/14\Z$ when base extended to $\Q_{A_4}$.


So what remains to show is that, except for $M =  7, 13, 14$, and $18$ as explained above, the combinations of full $2$-torsion and odd torsion that do not occur over $\QQ$ also do not occur over $\QQ_{A_4}$. We start by ruling out the case where $E$ has full 2-torsion over $\QQ_{A_4}$  and a point of order $5, 9, 13,$ or $21.$ In the first four cases, using the information in Table \ref{tab:strong_cyclic_groups} we can compute the discriminant of the generic curve with a point of this order and see that there are no rational numbers that make it a square. If $E$ has a point of order 21 over $\QQ_{A_4}$, then by \cite[Theorem 1]{najman-cubic}, $E$ is isomorphic to the curve with Cremona reference \href{http://www.lmfdb.org/EllipticCurve/Q/162b1}{\texttt{162b1}}. The curve \href{http://www.lmfdb.org/EllipticCurve/Q/162b1}{\texttt{162b1}} has discriminant $-2^3 3^4$, and since the discriminant is invariant modulo squares under $\Qbar$-isomorphism, this case is excluded by Lemma \ref{lem:strong_full_2}.

It only remains to show that there are no curves $E/\QQ$ with $E(\QQ_{A_4})_\tor\simeq \ZZ/M\ZZ$ for $M = 26$ or $42$, but this follows from the fact that there are no elliptic curves with $26$ or $42$ isogenies.
\end{proof}

\begin{example}\label{ex-Strong_A4_2,14}
An elliptic curve $E/\QQ$ has  $E(\QQ_{A_4})_\tor \simeq \ZZ/2\ZZ \oplus \ZZ/14\ZZ$ if and only if it is $\QQ$-isomorphic to
\begin{align*}
y^2 = x^3  &-27(t^2 - t + 7)^3(t^2 + t + 7)^3(t^4 + 5t^2 + 1)x \\
\phantom{y^2 =}& +54(t^2 - t + 7)^4(t^2 + t + 7)^4(t^8 + 14t^6 + 63t^4 + 70t^2 - 7)
\end{align*}
for some $t\in\QQ$.
\end{example}

\begin{example}
An elliptic curve $E/\QQ$ has  $E(\QQ_{A_4})_\tor \simeq \ZZ/2\ZZ \oplus \ZZ/6\ZZ$ if and only if it is $\QQ$-isomorphic to
\[
y^2 = x^3  -27(t^2+3)(t^2+27)^3 x + 54(t^2 + 27)^4(t^4 + 18t^2 - 27)
\]
for some $t\in\QQ$.
\end{example}

\begin{example}\label{ex-Strong_A4_18}
The generic elliptic curve with a point of order 18 over $\QQ_{A_4}$ is given by
\begin{equation*}
y^2+(t^3-2)xy+t^3y = x^3
\end{equation*}
For a particular value of $t$, the point of order 18 is defined over the splitting field of the polynomial $f(x) = x^3-t(t^3 + 3t - 2)x^2 + t^3(t^3 - 2)x + t^6$ which outside of a thin set defines a degree 3 cyclic extension of $\QQ$.
\end{example}

\subsection{Torsion over $\QQ(C_3^\infty)$}\label{sec:torsion-cubic} Comparing Theorem \ref{thm:main_strong} with \cite[Theorem 1]{najman-cubic}, we see that all of the groups that occur when base extending an elliptic curve to $\QQ_{A_4}$ also occur when base extending elliptic curves to a single cubic extension.

Inspecting the proof of Theorem \ref{thm:main_strong} we see that whenever we have growth in the torsion subgroup when extending to $\QQ_{A_4}$ it occurs when extending to at most 2 cyclic cubic extensions of $\QQ$. That is to say, for any elliptic curve $E/\QQ$, we have $E(\QQ_{A_4})_\tor = E(\QQ(C_3^\infty))_\tor$. With this observation we get the following corollary.

\begin{cor}\label{cor:final-corollary-cubic}
Let $E/\QQ$ be an elliptic curve. The torsion subgroup $E(\QQ(C_3^\infty))_\tor$ is finite and
\[
E(\QQ(C_3^\infty))_\tor\simeq
\begin{cases}
\ZZ/M\ZZ & \hbox{with } 1\leq M \leq 10 \hbox{ or } M = 12, 13, 14, 18, 21 \hbox{ or}\\
\ZZ/2\ZZ \oplus \ZZ/2M\ZZ & \hbox{with }1\leq M \leq 4 \hbox{ or } M=7.\\
\end{cases}
\]

\end{cor}

\begin{remark}It is also interesting to compare this to \cite[Theorem 4.1]{derickx-najman}, which classifies $K$-rational torsion structures of elliptic curves $E/K$ for $K$ a cyclic cubic field. The present question is similar to but essentially different from the one studied in \cite{najman-cubic}. The main difference is that the condition that $E$ has to be defined over $\QQ$ is dropped in \cite{derickx-najman}. Indeed, if $E$ is allowed to be defined over a cyclic cubic field $K$ instead of over $\QQ$ then there are the extra possibilities of $\ZZ/16\ZZ$, $\ZZ/2\ZZ \oplus \ZZ/10\ZZ$ and $\ZZ/2\ZZ \oplus \ZZ/12\ZZ$ for $E(K)_{\tor}$. The analogous question for $K =\Q(C_3^\infty)$ is one we do not know how to solve using our techniques, since this question cannot be solved solely by studying the action of $\mathrm{Gal}(\overline \QQ/\QQ)$ on $E(\overline \QQ)_{\tor}$. Indeed, when $E$ is defined over $\QQ(C_3^\infty)$ then there is only an action of $\mathrm{Gal}(\overline \QQ/\QQ(C_3^\infty))$.\end{remark}

\section*{Acknowledgments}
The authors would like to thank \'Alvaro Lozano-Robledo, Keith Conrad, and Filip Najman for useful conversations during the writing of this paper, as well as Enrique Gonz\'{a}lez-Jim\'{e}nez for helping with the writing in an early version of this paper. We are also deeply indebted to the referee and editor for catching typos and errors; their suggestions have greatly improved the quality of the paper.

\bibliographystyle{amsalpha}
\bibliography{references}

\end{document}